\documentclass[12pt]{amsart}

\usepackage{graphics,color,tikz,mathrsfs,amscd,eucal,mathabx,verbatim,booktabs}
\usepackage{graphicx} 
\usepackage{booktabs} 
\usepackage{amscd}
\usepackage{verbatim}
\usetikzlibrary{calc}
\usepackage{amssymb}

\numberwithin{equation}{section}
\numberwithin{table}{section}

\theoremstyle{plain}
\newtheorem{theorem}{Theorem}[section]
\newtheorem{lemma}{Lemma}[section]
\newtheorem{prop}{Proposition}[section]

\theoremstyle{remark}

\newtheorem{algorithm}{Algorithm}[section]

\newcommand{\Clement}{{C}l\'ement}

\newcommand{\Nedelec}{N{\'{e}}d{\'{e}}lec }

\newcommand{\dive}{\mathop\mathrm{div}}
\newcommand{\grad}{\ensuremath{\mathop{{\bf{grad}}}}}

\newcommand{\bcurl}{\mathop{\mathbf{curl}}}

\newcommand{\curl}{\mathrm{curl}}

\newcommand{\TT}{\mathscr{T}}

\newcommand{\ba}{{\boldsymbol{a}}}

\newcommand{\bB}{{\boldsymbol{B}}}
\newcommand{\bb}{{\boldsymbol{b}}}

\newcommand{\bC}{{\boldsymbol{C}}}
\newcommand{\bbc}{{\boldsymbol{c}}}

\newcommand{\bz}{{\boldsymbol{z}}}

\newcommand{\be}{{\boldsymbol{e}}}

\newcommand{\bn}{{\boldsymbol{n}}}

\newcommand{\bN}{{\boldsymbol{N}}}

\newcommand{\bpsi}{{\boldsymbol{\psi}}}
\newcommand{\bq}{{\boldsymbol{q}}}
\newcommand{\bxi}{{\boldsymbol{\xi}}}

\newcommand{\bF}{{\boldsymbol{F}}}

\newcommand{\bH}{{\boldsymbol{H}}}

\newcommand{\bL}{{\boldsymbol{L}}}

\newcommand{\bu}{{\boldsymbol{u}}}

\newcommand{\bv}{{\boldsymbol{v}}}
\newcommand{\bw}{{\boldsymbol{w}}}

\newcommand{\RRR}{{\mathbb{R}}}
\newcommand{\br}{{\boldsymbol{r}}}

\newcommand{\bt}{{\boldsymbol{t}}}
\newcommand{\bx}{{\boldsymbol{x}}}

\newcommand{\bX}{{\boldsymbol{X}}}

\newcommand{\bff}{{\boldsymbol{f}}}

\newcommand{\tn}{\textnormal}

\newcommand{\hdivk}{\mathbf{H}_r(\dive\nolimits^k,\Omega)}
\newcommand{\hcurlk}{\mathbf{H}_r(\bcurl\nolimits^k,\Omega)}
\newcommand{\Lr}{L^2_r(\Omega)}



\author{Minah Oh} 
\address{James Madison University, Department of Mathematics and Statistics,
  Harrisonburg, VA 22807.}\email{ohmx@jmu.edu}
  \thanks{The work of the author was partially supported by NSF grant number DMS-1913050.}

\title[Multigrid in Weighted H(div)]{Multigrid in H(div) on Axisymmetric Domains}



\begin{document}

\begin{abstract}
In this paper, we will construct and analyze a multigrid algorithm that can be applied to weighted $\bH(\dive)$ problems on a two-dimensional domain.
These problems arise after performing a dimension reduction to a three-dimensional axisymmetric $\bH(\dive)$ problem.
We will use recently developed Fourier finite element spaces that can be applied to axisymmetric $\bH(\dive)$ problems with general data.
We prove that if the axisymmetric domain is convex, then the multigrid V-cycle with modern smoothers will converge uniformly with respect to the meshsize.
\end{abstract}

\maketitle

\section{Introduction}
Let $\breve{\Omega}\subset \RRR^3$ be an axisymmetric domain, i.e., $\breve{\Omega}$ is obtained by rotating $\Omega\subset\RRR^2_+=\{(r,z)\in \RRR^2: r\geq0 \}$ around the $z$-axis. 
Throughout this paper, we will assume that $\breve{\Omega}$ is convex. 
The Hilbert space $\bH(\dive, \breve{\Omega})$ consists of square integrable vector-valued functions defined on $\breve{\Omega}$ whose divergence is also square integrable. The inner product on this space is given by 
\[
(\bu,\bv)_{L^2(\breve{\Omega})} + (\dive\bu,\dive\bv)_{L^2(\breve{\Omega})},
\] 
where $(\cdot,\cdot)_{L^2(\breve{\Omega})}$ denotes the usual $L^2$-inner product.
Consider the following axisymmetric $\bH(\dive)$-problem: find $\bu\in \bH(\dive, \breve{\Omega})$ such that
\begin{equation} \label{3d}
(\bu,\bv)_{L^2(\breve{\Omega})} + (\dive\bu,\dive\bv)_{L^2(\breve{\Omega})} = (\bF,\bv)_{L^2(\breve{\Omega})} \quad\tn{ for all } \bv\in \bH(\dive, \breve{\Omega}).
\end{equation} 
Numerical methods that can be applied to problems like (\ref{3d}) on general $3D$ domains (not necessarily axisymmetric) have many applications, see \cite[section 7]{mg:hdiv} and \cite{Lin:1997} for example, 
and multigrid methods for these problems have been constructed and studied in \cite{AFW:2000} and \cite{Hiptmair:hdiv}.

For problems defined on an axisymmetric domain such as (\ref{3d}), one can use a Fourier series decomposition to change the three-dimensional ($3D$) problem into a sequence of 
two-dimensional ($2D$) problems defined on the meridian domain $\Omega\subset\RRR^2_+$. Fourier Finite Element Methods (Fourier-FEMs) can be used to approximate each Fourier-mode of the solution $\bu$ by using a suitable FEM.
Such dimension reduction is an attractive feature considering computation time, but the resulting weighted $2D$ problems are quite different from the corresponding unweighted problems as we will see in the next section.
The appropriate weighted spaces include functions with singularities at the axis of rotation, 
so the analysis of such weighted problems require special attention. 
When the data function $\bF$ in (\ref{3d}) is independent of the rotational variable $\theta$ when written in terms of cylindrical coordinates, then the solution $\bu$ is also independent of the $\theta$-variable, and only the zero-th Fourier mode of 
$\bu$ is alive. In most applications, however, $\bF$ is dependent on the $\theta$-variable, so $\bu$ has non-vanishing higher order Fourier-modes. In this paper, we construct and analyze a multigrid algorithm that can be applied to 
weighted $2D$-problems arising from (\ref{3d}) that provide an approximate solution to each $k$-th Fourier mode of the exact solution $\bu$ for $|k|\geq 1$. 

Multigrid methods for axisymmetric $\bH(\bcurl)$ problems have been studied in the past. Multigrid for the azimuthal problem was analyzed in \cite{GP:2006}, and another multigrid analysis was done by using line relaxations in \cite{MR2003d:78042}.   
In \cite{CGO:2010}, a multigrid analysis for the V-cycle algorithm was done for the meridian problem by using the smoothers in \cite{AFW:2000} and \cite{H:1998}.
To our knowledge, multigrid methods for axisymmetric $\bH(\dive)$ problems have not been studied yet.
All of the above mentioned papers are assuming that not only the $3D$ domain is axisymmetric but also the data given in the problem is axisymmetric as well. 
One reason for this was perhaps the lack of commuting projections associated with axisymmetric problems with general data. 
This difficulty was overcome in \cite{O:2015} where the author constructed a new family of Fourier-finite element spaces whose interpolation operators satisfy a commuting diagram property.
Therefore, in this paper, we will use the $\bH_r(\dive^k)$-conforming finite element space for $|k|\geq 1$ constructed in \cite{O:2015} to construct and analyze a multigrid V-cycle that can be applied to 
weighted $2D$-problems arising from (\ref{3d}). 
We will follow the framework of \cite{AFW:2000} for the multigrid analysis.

This paper is organized as follows: in section~\ref{prelim}, we will give an overview on Fourier-FEMs and state the weighted problem of interest.  
We will also summarize the definitions of some needed weighted spaces and a family of Fourier finite element spaces along with commuting projections onto those spaces. 
In section~\ref{concrete}, we prove better error estimates for the commuting projections constructed in \cite{O:2019}. The new ideas taken here is in the construction and use of operators that have appropriate error estimates in the weighted $L^2$-space
with the measure $r^3drdz$ instead of the usual $rdrdz$. This will help us deal with the multiple $1/r$ terms appearing in the interpolation operators used  in \cite{O:2015}.
In section~\ref{mixed}, a weighted mixed formulation that will be helpful in the multigrid analysis will be studied, and
in the following two sections, the multigrid V-cycle algorithm will be introduced and analyzed. Finally, in section~\ref{numerics} numerical results that support the mathematical theory established in this paper are provided
followed by a section with concluding remarks.
Some technical proofs are included in the Appendix (section~\ref{appendix}) to improve the readability of the paper.

\section{Preliminaries} \label{prelim}


In this section, we summarize definitions of weighted spaces as well as Fourier finite element spaces. We will also state the weighted problem of interest.

If $\breve{f}$ is a function defined on $\breve{\Omega}$ that is independent of the $\theta$-variable, then 
\[
\int\int\int_{\breve{\Omega}} \breve{f}(r,\theta,z)^2 dV = 2\pi\int\int_\Omega f(r,z)^2 rdrdz,
\] 
where $\breve{f}$ and $f$ are functions defined on $\breve{\Omega}$ and $\Omega$ respectively with the same formula.
Therefore, we are interested in the weighted $L^2$-space defined by
\[
L^2_r(\Omega)=\{u: \int\int_\Omega u(r,z)^2 rdrdz < \infty \}.
\]
This is a Hilbert space with the inner-product being 
\[
(u,v)_r = \int\int_\Omega uv rdrdz,
\]
and the induced norm will be denoted by $\| \cdot \|_{L^2_r(\Omega)}$. In general, we will use $\| \cdot \|_X$ to denote the norm of the Hilbert space $X$.
Notice that
\[
u\in L^2_r(\Omega) \tn{ if and only if } \dfrac{u}{r} \in L^2_{r^3}(\Omega),
\]
where 
\[
L^2_{r^3}(\Omega)=\{u: \int\int_\Omega u(r,z)^2 r^3drdz < \infty \}.
\]
In general, we may define a weighted $L^2$-space in the following way:
\[
L^2_\alpha(\Omega) = \{u: \int\int_\Omega u(r,z)^2 r^\alpha drdz < \infty \}
\]
with the associated norm
\[
\| u \|_{L^2_\alpha(\Omega} = (\int\int_\Omega u(r,z)^2 r^\alpha drdz)^{1/2}.
\]
In this paper, we will be mainly using $\alpha=1$, but $\alpha=3$ and $\alpha=-1$ will be used in some places. 
Since these are the only three $\alpha$ values that will be used, we will simply write $L^2_r(\Omega)$, $L^2_{r^3}(\Omega)$, and $L^2_{1/r}(\Omega)$ respectively to denote these spaces. 

Let 
\[
\grad\nolimits_{rz}u=(\dfrac{\partial u}{\partial r}, \dfrac{\partial u}{\partial z})^T.
\]
Then, we define
\begin{align*}
H^1_r(\Omega) &= \{u\in L^2_r(\Omega): \grad\nolimits_{rz}u\in L^2_r(\Omega) \}, \\
\tilde{H}_r^1(\Omega) &= H^1_r(\Omega) \cap L^2_{1/r}(\Omega),
\end{align*}
and the associated norm for these spaces are
\begin{align*}
\| u \|_{H^1_r(\Omega)} &= (\| u \|_{L^2_r(\Omega)}^2 + \|\grad\nolimits_{rz} u \|^2_{L^2_r(\Omega)})^{1/2}, \\
\| u \|_{\tilde{H}^1_r(\Omega)} &=  (\| u \|_{H^1_r(\Omega)}^2 + \| u \|^2_{L^2_{1/r}(\Omega)})^{1/2}.
\end{align*}
We will use $\Gamma_1$ to denote the boundary of $\Omega$ that is not on the axis of symmetry, i.e.,
the rotation of $\Gamma_1$ returns $\partial\breve{\Omega}$, and $\Gamma_0$ to denote the part of $\partial{\Omega}$ that is on the axis of symmetry, i.e., $\Gamma_0=\partial\Omega \backslash \Gamma_1$.
Then $H^1_{r,0}(\Omega)$ denotes the closed subspace of $H^1_r(\Omega)$ with vanishing trace on $\Gamma_1$. 
In general, we will use $H^l_r(\Omega)$ to denote the subspace of $L^2_r(\Omega)$ that consists of functions whose distributional derivatives of order $l$
 and under also belong in $L^2_r(\Omega)$. 
 Furthermore, $H^1_{1/r}(\Omega)$ will be used to denote the space of functions in $L^2_{1/r}(\Omega)$ whose gradient also belongs in $L^2_{1/r}(\Omega)$.
 Similarly, $H^1_{r^3}(\Omega)$ denotes the subspace of $L^2_{r^3}(\Omega)$ whose gradient is also in $L^2_{r^3}(\Omega)$.
 We will use boldface to denote vector-valued functions as well as functions spaces consisting of vector-valued functions. For simplicity,
 $\partial_r$ will be used instead of $\dfrac{\partial}{\partial r}$, etc. For $\bv=(v_r,v_\theta,v_z)^T$, we use $\bv_{rz}$ to denote $(v_r,v_z)^T$. 

Many authors have previously studied axisymmetric problems with general data through a Fourier series decomposition.
(See \cite{BDM:1999, Heinrich:1996, Boniface:2005, Nkemzi:2007, FSCM_Part1, FSCM_Part2, FSCM_Maxwell, O:2019} for example.)
Since each Fourier mode is obtained by taking an integral with respect to $\theta$, each Fourier mode is only dependent on variables $r$ and $z$,
and thus by using the axial symmetry of the $3D$ domain $\breve{\Omega}$ and a truncated partial Fourier series, one can reduce the $3D$ problem into $N$ $2D$ problems. 
We use the term Fourier-FEMs when each Fourier-mode of the solution is approximated by using a suitable FEM. 

For scalar-valued functions, the Fourier series decomposition takes the following form:
\[
u = u_0 + \sum_{k=1}^\infty u_k\cos k\theta + \sum_{k=1}^\infty u_{-k}\sin k\theta.
\]
For a vector-valued function, first write 
$\bu=u_r\be_r + u_\theta\be_\theta + u_z\be_z$ by using the cylindrical basis $\be_r$, $\be_\theta$, and $\be_z$.
Then, $\bu=\bu^s+\bu^a$ where
\begin{equation} \label{vector_fourier}
\begin{aligned}
\bu^s &= 
\begin{pmatrix}
u_r^0 \\
0 \\
u_z^0
\end{pmatrix}+ 
\sum_{k=1}^\infty
\begin{pmatrix}
u_r^k \cos k\theta \\
u_\theta^k \sin k\theta \\
u_z^k \cos k\theta
\end{pmatrix}, \\
\bu^a &= 
\begin{pmatrix}
0 \\
u_\theta^0 \\
0
\end{pmatrix}+ 
\sum_{k=1}^\infty
\begin{pmatrix}
u_r^{-k} \sin k\theta \\
u_\theta^{-k} \cos k\theta \\
u_z^{-k} \sin k\theta
\end{pmatrix}.
\end{aligned}
\end{equation}

Next, consider the usual divergence operator in cylindrical coordinates:
\begin{equation*} 
\begin{aligned}
\dive \bu &= \frac{1}{r}\partial_r (ru_r) + \frac{1}{r}\partial_\theta u_\theta + \partial_zu_z.
\end{aligned}
\end{equation*}
Now consider applying this operator to (\ref{vector_fourier}). 
Then, each $k$-th order Fourier mode in $\dive\bu$ decouple from one another in a weak formulation, and the resulting 
divergence formula that affects the $k$-th Fourier mode is
\[
\dive\nolimits_{rz}\nolimits^{k} 
\begin{bmatrix}
       u_r \\
       u_\theta \\
       u_z           
     \end{bmatrix}
= 
\partial_ru_r + \dfrac{u_r-ku_\theta}{r} + \partial_zu_z.
\]
We then define the following weighted $\bH(\dive)$ space for the $k$-th Fourier mode:
\[
\bH_r(\dive\nolimits^k,\Omega) = \{ \bu\in \bL^2_r(\Omega): \dive\nolimits_{rz}\nolimits^{k} \bu\in L^2_r(\Omega) \}.
\]
This is a Hilbert space with the inner product being
\begin{equation*} 
\Lambda^{k}(\bu,\bv)= (\bu,\bv)_r + (\dive\nolimits_{rz}\nolimits^{k} \bu, \dive\nolimits_{rz}\nolimits^{k} \bv)_r.
\end{equation*}
Let $\bC_h$ be a finite element subspace of $\bH_r(\dive^k,\Omega)$ that will be introduced shortly.  Define
$\Lambda_h^k: \bC_h \rightarrow \bC_h$ by
\begin{equation} \label{problem}
(\Lambda_h^{k}\bu_h,\bv_h)_r = \Lambda^{k}(\bu_h,\bv_h) \quad\tn{ for all } \bu_h, \bv_h\in \bC_h.
\end{equation}
In this paper, we construct and analyze a multigrid algorithm that can be applied to (\ref{problem}).
The analysis done in this paper holds true for any fixed integer $|k| \geq 1$,
so we assume that the Fourier-mode $k$ is fixed. For simplicity of notation, we will write $\Lambda$ instead of $\Lambda^k$.
The norm induced by $\Lambda(\cdot,\cdot)$ is denoted by $\| \cdot \|_{\Lambda}$.

Through a similar process, one also obtains the following $\grad$ and $\bcurl$ formulas that affects the $k$-th Fourier mode:
\begin{align*}
\grad\nolimits_{rz}\nolimits^{k} u 
&= \begin{bmatrix}
      \partial_r u\\
      -\frac{k}{r}u \\
       \partial_z u   
     \end{bmatrix}, \\
\bcurl\nolimits_{rz}\nolimits^{k} 
\begin{bmatrix}
       u_r \\
       u_\theta \\
       u_z           
     \end{bmatrix}
&= 
\begin{bmatrix}
       -(\frac{k}{r}u_z + \partial_z u_\theta)\\
       \partial_z u_r - \partial_r u_z \\
        \frac{ku_r+u_\theta}{r} + \partial_r u_\theta           
     \end{bmatrix},
\end{align*}
and we get the following Hilbert spaces:
\begin{align*}
H_r(\grad\nolimits_{}\nolimits^{k},\Omega) &= \{u\in L^2_r(\Omega): \grad\nolimits_{rz}\nolimits^{k} u \in \bL^2_r(\Omega) \}, \\
\bH_r(\bcurl\nolimits^k,\Omega) &= \{\bu\in \bL^2_r(\Omega): \bcurl\nolimits_{rz}\nolimits^{k}\bu\in \bL^2_r(\Omega) \}.
\end{align*}

Next, let 
\begin{align*}
\grad\nolimits_{rz}\nolimits^{k*} u 
&= \begin{bmatrix}
      \partial_r u\\
      \frac{k}{r}u \\
       \partial_z u   
     \end{bmatrix}, \\
\bcurl\nolimits_{rz}\nolimits^{k*} 
\begin{bmatrix}
       u_r \\
       u_\theta \\
       u_z           
     \end{bmatrix}
&= 
\begin{bmatrix}
       \frac{k}{r}u_z - \partial_z u_\theta \\
       \partial_z u_r - \partial_r u_z \\
        \frac{-ku_r+u_\theta}{r} + \partial_r u_\theta         
     \end{bmatrix},
\end{align*}
and define
\begin{align*}
H_{r,0}(\grad\nolimits_{}\nolimits^{k*},\Omega) &= \{u\in L^2_r(\Omega): \grad\nolimits_{rz}\nolimits^{k*} u \in \bL^2_r(\Omega) \tn{ and } u=0 \tn{ on } \Gamma_1 \}, \\
\bH_{r,0}(\bcurl\nolimits^{k*},\Omega) &= \{\bu\in \bL^2_r(\Omega): \bcurl\nolimits_{rz}\nolimits^{k*}\bu\in \bL^2_r(\Omega) \tn{ and }  \bu_{rz}\cdot\bt=0, u_\theta=0 \tn{ on } \Gamma_1 \}.
\end{align*}
As usual, $\bt$ denotes the unit tangent vector along $\Gamma_1$.

Then $\grad\nolimits_{rz}\nolimits^{k*}$ and $\bcurl\nolimits_{rz}\nolimits^{k*}$ are adjoints operators of $-\dive\nolimits_{rz}\nolimits^{k}$ and $\bcurl\nolimits_{rz}\nolimits^{k}$ respectively \cite[Theorem 7.1]{O:2019} in the following sense: 
\begin{align*}
(-\dive\nolimits_{rz}\nolimits^{k}\bu,\bv)_r &= (\bu,\grad\nolimits_{rz}\nolimits^{k*}\bv)_r &&\tn{ for all } \bu\in \bH_r(\dive\nolimits^k,\Omega), \bv\in H_{r,0}(\grad\nolimits_{}\nolimits^{k*},\Omega), \\
(\bcurl\nolimits_{rz}\nolimits^{k}\bu,\bv)_r &= (\bu, \bcurl\nolimits_{rz}\nolimits^{k*}\bv)_r &&\tn{ for all } \bu\in \bH_r(\bcurl\nolimits^k,\Omega), \bv\in \bH_{r,0}(\bcurl\nolimits_{}\nolimits^{k*},\Omega). \\
\end{align*}

Next, let $\TT_h$ be a finite element triangulation of $\Omega$ that satisfies the usual geometrical conformity conditions~\cite{fem:ciarlet}.
We now summarize the family of Fourier finite element spaces constructed in \cite{O:2015}. 
First define the following polynomial spaces:
\begin{align*}
A_1 &= \left\{ \alpha_1r+\alpha_2r^2+\alpha_3rz: \alpha_i\in\RRR \text{ for } 1\leq i\leq 3 \right\}, \\
\bB_1 &= \left\{ 
\begin{pmatrix}
       \beta_1 + \beta_4 r + \beta_3 z - \beta_6 rz \\
       -k\beta_1 + \beta_2 r - k\beta_3 z \\
       \beta_5 r + \beta_6 r^2
     \end{pmatrix}: \beta_i\in\RRR \text{ for } 1\leq i \leq 6
\right\}, \\
\bC_1 &= \left\{ 
\begin{pmatrix}
      k \gamma_1 + \gamma_2 r \\
      \gamma_1 + \gamma_3 r \\
       \gamma_4 + \gamma_2 z
     \end{pmatrix}: \gamma_i\in\RRR \text{ for } 1\leq i \leq 4 
\right\}. 
\end{align*}
We are interested in the following Fourier finite element spaces: 
\begin{equation} \label{fem_spaces}
\begin{aligned}
A_h &= \left\{u\in H_r( \grad\nolimits^k , \Omega): u|_K\in A_1 \text{ for all } K\in \TT_h  \right\}, \\
\bB_h &= \left\{\bu\in \bH_r(\bcurl\nolimits^k , \Omega): \bu|_K\in \bB_1 \text{ for all } K\in \TT_h  \right\}, \\
\bC_h &= \left\{\bu\in \bH_r(\dive\nolimits^k, \Omega): \bu|_K\in \bC_1 \text{ for all } K\in \TT_h  \right\}, \\
D_h &= \left\{u\in L^2_r(\Omega): u|_K \text{ is constant for all } K\in \TT_h  \right\}.
\end{aligned}
\end{equation}
Note that $\bB_h$ and $\bC_h$ are dependent on $k$.
It was proved in \cite[Theorem 4.1]{O:2015} that these Fourier finite element spaces are conforming and the corresponding interpolation operators satisfy the
so-called commuting diagram property with error estimates. We note that $\bB_h$ was constructed separately in \cite{GL:1991,Lacoste:2000}.
Furthermore in \cite{O:2019}, a set of projectors $\Pi_h^{g,k}, \Pi_h^{c,k}, \Pi_h^{d,k},$ and $\Pi_h^{o,k}$ that form a 
W-bounded cochain projection was constructed for this family of Fourier finite element spaces.
The following theorem can be found in \cite[Theorem 5.1]{O:2019}.

\begin{theorem} \label{ProjSet}
\label{thm:main}
There exists projections $\Pi_h^{g,k}, \Pi_h^{c,k}, \Pi_h^{d,k},$ and $\Pi_h^{o,k}$
that are continuous for all functions in $L^2_r(\Omega)$ (or $\bL^2_r(\Omega)$) that satisfy the following properties:
\begin{enumerate}
\item \label{comm} {\em Commutativity.} The operators make the following diagram commute:
\begin{equation*}
\begin{CD}
L^2_r(\Omega)   @>\mathbf{grad}^k_{rz}>> \bL^2_r(\Omega)   @>\bcurl\nolimits^k\nolimits_{rz}>> \bL^2_r(\Omega)   @>\dive^k_{rz}>>   L^{2}_{r}(\Omega) \\ 
@VV \Pi_h^{g,k} V                                   @VV \Pi_h^{c,k} V                           @VV \Pi_h^{d,k} V @ VV \Pi_h^{o,k} V  \\
A_h @> \grad_{rz}^k >> \bB_h @> \bcurl\nolimits^k\nolimits_{rz} >> \bC_h @> \dive_{rz}^k >> D_h
\end{CD}
\end{equation*}

\item \label{item:apprx} 
  {\em Approximation.} 
\begin{align*}
  \left\| u - \Pi^{g,k}_hu \right\|_{L^2_r(\Omega)}  &\leq C\inf_{u_h\in A_h}\left\| u - u_h \right\|_{L^2_r(\Omega)} , \\
  \left\| \bu - \Pi^{c,k}_h\bu \right\|_{L^2_r(\Omega)}  &\leq C\inf_{\bu_h\in \bB_h}\left\| \bu - \bu_h \right\|_{L^2_r(\Omega)} , \\ 
   \left\| \bu - \Pi^{d,k}_h\bu \right\|_{L^2_r(\Omega)}  &\leq C\inf_{\bu_h\in \bC_h}\left\| \bu - \bu_h \right\|_{L^2_r(\Omega)} , \\
   \left\| u - \Pi^{o,k}_h u \right\|_{L^2_r(\Omega)}  &\leq C\inf_{u_h\in D_h}\left\| u - u_h \right\|_{L^2_r(\Omega)} . 
\end{align*}
\end{enumerate}
\end{theorem}

In this paper, we will mainly use the projection $\Pi^{d,k}_h$ onto $\bC_h$. While the original interpolation operator used in the construction of $\bC_h$ in \cite{O:2015}
satisfy error estimates, that interpolation operator requires more regularity on the function than necessary. Therefore, in the next section, we will construct another projection onto $\bC_h$ 
to show that
\[
\inf_{\bu_h\in \bC_h}\left\| \bu - \bu_h \right\|_{L^2_r(\Omega)} \leq Ch(|\bu_{rz}|^2_{H^1_r(\Omega)}+ \| ku_\theta -u_r \|_{\tilde{H}^1_r(\Omega)}).
\]

\section{Error Estimates for Commuting Projections} \label{concrete}

In this section, we will construct another interpolation operator onto $\bC_h$ to obtain better concrete error estimate for $\Pi^{d,k}_h$. 

The following result can be found in \cite[Proposition 3.3]{O:2015}. 
\begin{prop} \label{div_fem}

The following finite element $(\Sigma^k, K, P^k)$ is unisolvent and conforming in $\bH_r(\dive^k, \Omega)$.

\begin{itemize}

\item $K$: triangle with vertices $a_i$, edges $e_i$, and normal vectors $\bn_i$, $1\leq i\leq 3$.

\item $P^k$: space of polynomials defined by 
$$ \bu=    \begin{pmatrix}
       k\gamma_1 + \gamma_2 r \\
       \gamma_1 + \gamma_3 r \\
       \gamma_4 + \gamma_2 z
     \end{pmatrix}. 
$$

\item $\Sigma^k$: set of linear forms, for $1\leq i \leq 3$
\begin{align*}
\sigma_{K} &: \bu \mapsto \dfrac{1}{|K|}\int_K \dfrac{ku_\theta - u_r}{r},  \\ 
\sigma_{e_i} &:        \bu \mapsto \int_{e_i} 
     \begin{pmatrix}
       u_r \\
       u_z     
       \end{pmatrix} \cdot \bn_i.  
\end{align*}
\end{itemize}
\end{prop}
Notice that $\bu_{rz}=(u_r,u_z)^T$ is being projected onto the lowest order Raviart Thomas space \cite{RT:1977},
while $\dfrac{ku_\theta-u_r}{r}$ is being projected onto the piecewise constant space. 
The degrees of freedom used above to do so does not take into consideration that the functions that we are applying the corresponding interpolation operator (denoted by $I_h^{d,k}$)   
come from weighted function spaces. Therefore, the error estimates one gets for $I_h^{d,k}$ in \cite[Theorem 4.1]{O:2015} requires $u_r, u_z$, and $\dfrac{ku_\theta - u_r}{r}$ to be in
the space $H^2_r(\Omega)$ which is known to be continuously embedded in $L^2(\Omega)$ \cite{MR:1982}. 

In \cite[Lemma 5.3]{CGP:2008} an interpolation operator onto the lowest order \Nedelec space  \cite{fem:nedelec} that is continuous on $\bH^1_r(\Omega)$ was constructed,
and a projection onto the same space that is continuous on $\bL_r^2(\Omega)$ was constructed in \cite{GO:2012} as well.
Employing a similar idea, one might consider taking the $L^2_r$-orthogonal projection of $\dfrac{ku_\theta-u_r}{r}$ instead of the standard $L^2$-orthogonal projection in Proposition~\ref{div_fem},
but this does not solve the issue of the regularity condition being posed on  $\dfrac{ku_\theta-u_r}{r}$ instead of $ku_\theta-u_r$ for example.
As demonstrated in the proof of the following theorem, the trick is to use the weight $r^3$ instead of the usual weight $r$.
This is because as in Proposition~\ref{div_fem}, a projection will be applied to  $\dfrac{ku_\theta - u_r}{r}$ that involves a 
$\dfrac{1}{r}$-terms and 
\[
v\in L^2_r(\Omega) \quad\tn{ if and only if }\quad \dfrac{v}{r}\in L^2_{r^3}(\Omega).
\]

Let $\bH_{r,k}^{1\star}(\Omega)$ denote the space $\{ \bu\in \bH^1_r(\Omega): ku_\theta -u_r\in L^2_{1/r}(\Omega) \}$.
Then we have the following theorem.
\begin{theorem} \label{concreteD} For all $\bu\in \bH^{1\star}_r(\Omega)$, the commuting projection $\Pi^{d,k}_h$ satisfies
\[
\| \bu - \Pi^{d,k}_h\bu \|_{L^2_r(\Omega)}  \leq Ch(|\bu_{rz}|^2_{H^1_r(\Omega)}+ \| ku_\theta -u_r \|_{\tilde{H}^1_r(\Omega)}).
\]
\end{theorem}

\begin{proof}
By Theorem~\ref{ProjSet} item~(\ref{item:apprx}), it suffices to prove that
\[
\inf_{\bu_h\in \bC_h}\left\| \bu - \bu_h \right\|_{L^2_r(\Omega)} \leq Ch(|\bu_{rz}|^2_{H^1_r(\Omega)}+ \| ku_\theta -u_r \|_{\tilde{H}^1_r(\Omega)}).
\]


Define the set of global degrees of freedom of $\bC_h$ as 
\begin{equation} \label{newC2}
\bar{\sigma}_K (\bu) = \dfrac{\int_K r^3 \cdot \dfrac{ku_\theta - u_r}{r}  dA}{  \int_K r^3 dA} \quad\quad\tn{ for all mesh triangles } K
\end{equation}
and
\begin{equation} \label{newC}
\begin{aligned}
\bar{\sigma}_{\ba}(\bu) &=  \int_{e(\ba)} \bu_{rz}\cdot\bn r ds &&\tn{ for all mesh vertices } \ba \in\bar{\Gamma}_0, \\
\bar{\sigma}_e (\bu) &= \int_e \bu_{rz}\cdot\bn ds &&\tn{ for all mesh edges } e \cap \bar{\Gamma}_0 = \emptyset , \\
\bar{\sigma}_{K_0} (\bu) &= \int_{K_0} r \dive\nolimits_{rz} \bu_{rz}  dA &&\tn{ for all mesh triangles } K_0 \cap \bar{\Gamma}_0 \ne \emptyset,
\end{aligned}
\end{equation}
where $e(\ba)$ is an edge associated with a vertex $\ba$ on $\Gamma_0$ that is not on $\Gamma_0$, and
\[
\dive\nolimits_{rz} \bu = \partial_r u_r + \partial_z u_z.
\]
Note that the degrees of freedom (\ref{newC}) are the same as the ones used in \cite[Lemma 5.3]{CGP:2008} for the lowest order \Nedelec space 
when viewing the lowest order Raviart Thomas space as the rotated lowest order \Nedelec space. 
The linear functionals (\ref{newC}) are continuous linear functionals on $\bH^1_r(\Omega)$ and form unisolvent degrees of freedom for the lowest order Raviart Thomas space \cite[Proposition 5.2]{CGP:2008}.

We clarify the difference between the degrees of freedom (\ref{newC2}) and (\ref{newC}) and the ones used in Proposition~\ref{div_fem}.
First of all, in (\ref{newC2}), the $L^2_{r^3}$-orthogonal projection of $\dfrac{kv_\theta - v_r}{r}$ is used while
the $L^2$-orthogonal projection of $\dfrac{kv_\theta - v_r}{r}$ is used in Proposition~\ref{div_fem}. Secondly, in (\ref{newC}), the degrees of freedom used in \cite[Lemma 5.3]{CGP:2008} are used for $\bu_{rz}$ 
instead of the standard non-weighted lowest order Raviart-Thomas degrees of freedom are used in Proposition~\ref{div_fem}.

We define $\tilde{\Pi}^{d,k}_h: \bH_{r}(\dive^n,\Omega) \rightarrow \bC_h$ as
\[
\bar{\sigma} (\tilde{\Pi}^{d,k}_h\bv) = \bar{\sigma}(\bv) 
\] 
for all degrees of freedom $\bar{\sigma}$ in (\ref{newC2}) and (\ref{newC}). In other words, for each triangle $K\in\TT_h$, 
\begin{equation}  \label{pi_div}
\begin{aligned} 
            \tilde{\Pi}^{d,k}_h \bv \begin{pmatrix}
       v_r \\
       v_\theta \\
       v_z          
     \end{pmatrix} |_K &=
     \bar{\sigma}_K (\bv)
     \begin{pmatrix}
        0 \\
        \dfrac{r}{k}\chi_K \\
        0          
     \end{pmatrix}  +
     \sum_{i=1}^3
\bar{\sigma}_i (\bv) 
         \begin{pmatrix}
        \xi_i^r \\
        \dfrac{1}{k}\xi_i^r \\
        \xi_i^z    
       \end{pmatrix},
\end{aligned}
\end{equation}
where $\bar{\sigma}_i$ denotes the local degrees of freedom corresponding to (\ref{newC}),
$\bxi_i=(\xi_i^r,\xi_i^z)^T$ is the local basis for the lowest order Raviart Thomas space associated with $e_i$, and 
$\chi_K$ denotes the constant function one on triangle $K$.

For any edge $e$ in $\TT_h$, let $\Delta_e$ denote the union of two triangles that share $e$ as a common edge, and let $h_e$ denote its diameter.
Then from \cite[Lemma 5]{BBD:2006} we get
\[
\inf_{q \in P_0(\Delta_e)} \|v-q\|_{L^2_{r}(\Delta_e)} \leq Ch_e|v|_{H^1_{r}(\Delta_e)},
\]
where $P_0=\{a: a \in \RRR \}$.
By replacing $r$ by $r^3$ in the proof of \cite[Lemma 5]{BBD:2006} before applying Young's inequality, 
this result extends to the following:
\begin{equation} \label{r3}
\inf_{q \in P_0(\Delta_e)} \|v-q\|_{L^2_{r^3}(\Delta_e)} \leq Ch_e|v|_{H^1_{r^3}(\Delta_e)}.
\end{equation}
As in \cite[Lemma 5]{BBD:2006}, this remains true when $\Delta_e$ is replaced by any triangle $K$ as well:
\begin{equation} \label{r3_1}
\inf_{q \in P_0(K)} \|v-q\|_{L^2_{r^3}(K)} \leq Ch_K|v|_{H^1_{r^3}(K)},
\end{equation}
where $h_K$ is the diameter of $K$.

Next, we say that a triangle is type 1 or type 2 if it intersects $\Gamma_0$ at one vertex or two vertices resepctively, and
we call a triangle type 3 if it does not intersect $\Gamma_0$ at all. 
If $K$ is a type 1 triangle, let $D_K$ denote the union of all triangles connected to the one vertex of $K$ that is on $\Gamma_0$.
If $K$ is of type 2, then $D_K$ denotes the union of the two vertex patches of the two vertices of $K$ that is on $\Gamma_0$.
For type 3 triangles, $D_K$ is simply equal to $K$. 
Then, for any triangle $K\in\TT_h$, we have
\begin{equation} \label{rrr3}
\begin{aligned}
\left\| \bv - \tilde{\Pi}_h^{d,k}\bv \right\|_{r,K}^2
&= \left\| \bv_{rz} -   \sum_{i=1}^3\bar{\sigma}_i (\bv) \bxi_i  \right\|^2_{r,K} 
+ \left\| v_\theta -  \bar{\sigma}_K (\bv) \dfrac{r}{k}\chi_K -  \sum_{i=1}^3\bar{\sigma}_i (\bv) \dfrac{1}{k}\xi_i^r  \right\|_{r,K}^2, \\
&\leq C (h^2|\bv_{rz}|^2_{H^1_r(D_K)} 
+ \left\| \dfrac{1}{k}(kv_\theta -  \bar{\sigma}_K (\bv) r\chi_K) -   \dfrac{1}{k}\sum_{i=1}^3\bar{\sigma}_i (\bv)\xi_i^r  \right\|_{r,K}^2)  \quad \tn{ by \cite[Lemma 5.3]{CGP:2008}, }\\
&\leq C (h^2|\bv_{rz}|^2_{H^1_r(D_K)} 
+ \left\| (kv_\theta -v_r)  -  \bar{\sigma}_K (\bv) r\chi_K + v_r -  \sum_{i=1}^3\bar{\sigma}_i (\bv) \xi_i^r  \right\|_{r,K}^2), \\
&\leq C (h^2|\bv_{rz}|^2_{H^1_r(D_K)} 
+ \left\| (kv_\theta -v_r)  -  \bar{\sigma}_K (\bv) r\chi_K \right\|_{r,K}^2),\\
&= C (h^2|\bv_{rz}|^2_{H^1_r(D_K)} 
+ \left\| r(\dfrac{kv_\theta -v_r}{r}  -  \bar{\sigma}_K (\bv) \chi_K) \right\|_{r,K}^2), \\
&= C (h^2|\bv_{rz}|^2_{H^1_r(D_K)} 
+ \left\| \dfrac{kv_\theta -v_r}{r}  -  \bar{\sigma}_K (\bv) \chi_K \right\|_{L^2_{r^3}(K)}^2), \\
&\leq C(h^2|\bv_{rz}|^2_{H^1_r(D_K)}  + Ch^2|\dfrac{kv_\theta -v_r}{r}|^2_{H^1_{r^3}(K)}) \quad\tn{ by (\ref{r3_1}). }
\end{aligned}
\end{equation}
The last inequality is using the fact that 
the error of the $L^2_{r^3}$-orthogonal projection is bounded by the best approximation error in the $L_{r^3}^2$-norm.

Direct calculation shows that
\begin{equation} \label{r4}
|\dfrac{kv_\theta -v_r}{r}|_{H^1_{r^3}(K)} \leq  \| kv_\theta -v_r \|_{\tilde{H}^1_r(K)}.
\end{equation}
Therefore, by (\ref{rrr3}) and (\ref{r4}), and summing over all triangles as usual, we conclude that
\begin{equation} \label{Pi3Est}
\left\| \bv - \tilde{\Pi}_h^{d,k}\bv \right\|_{L^2_r(\Omega)} \leq Ch(|\bv_{rz}|_{H^1_r(\Omega)}+ \| kv_\theta -v_r \|_{\tilde{H}^1_r(\Omega)}).
\end{equation}

\end{proof}

Next, let 
\[
\bH_{r,k}^{1\diamond}(\Omega)=\{\bv\in \bH^1_r(\Omega): kv_z, kv_r+v_\theta \in L^2_{1/r}(\Omega) \}.
\]
Then, by using a similar idea as in the proof of Theorem~\ref{concreteD} and Cl\'ement operators \cite{Clement:1975, MR:1982, BBD:2006}, 
we get the following error estimate for the commuting projector $\Pi_h^{c,k}$. 
\begin{theorem} \label{cc}
For all $\bv\in \bH_{r,k}^{1\diamond}(\Omega)$ we have 
\[
\left\| \bv - \Pi_h^{c,k}\bv \right\|_{L^2_r(\Omega)} \leq Ch(|v_\theta|_{H^1_r(\Omega)} + \|kv_z\|_{\tilde{H}^1_r(\Omega)}  +  \|kv_r+v_\theta\|_{\tilde{H}^1_r(\Omega)}).
\]
\end{theorem}
The proof of Theorem~\ref{cc} can be found in the Appendix. It uses the $L^2_{r^3}$-orthogonal projection when constructing a weighted Cl\'ement-type operator.

\section{A Weighted Mixed Problem} \label{mixed}

For the multigrid analysis in section~\ref{multigrid2}, we need to study the following weighted mixed problem:

find $(\bz,p)\in \bH_r(\dive^k,\Omega)\times L^2_r(\Omega)$ such that
\begin{equation} \label{mixed1}
\begin{aligned}
(\bz,\bw)_r - (\dive\nolimits_{rz}^k\bw, p)_r &= 0 &&\tn{ for all } \bw\in\bH_r(\dive\nolimits^k,\Omega), \\
(\dive\nolimits_{rz}^k\bz,s)_r &= (f,s)_r &&\tn{ for all } s\in L^2_r(\Omega).
\end{aligned}
\end{equation}
This mixed problem is the weighted Poisson equation with Dirichlet boundary conditions, and $p$ is the solution of
\begin{equation} \label{cts1}
\begin{aligned}
-\Delta_{(k)} p &= f &&\tn{ in } \Omega, \\
p &= 0 &&\tn{ on } \Gamma_1,
\end{aligned}
\end{equation}  
where $\Delta_{(k)}=\dive\nolimits_{rz}^k\grad\nolimits_{rz}^{k*}$. 
The mixed problem (\ref{mixed1}) has been studied in \cite{O:2019} as one of the axisymmetric Hodge Laplacian problems. 
It follows from \cite[Theorem 3.1]{O:2019} that (\ref{mixed1}) is well-posed and that
\[
\| \bz \|_{\bH_r(\dive^k,\Omega)} + \| p \|_{L^2_r(\Omega)} \leq C\| f \|_{L^2_r(\Omega)}.
\]
In \cite{FSCM_Part2}, more detailed regularity results were given for the solution of (\ref{cts1}). 
For our multigrid analysis, we further prove the following regularity result.
\begin{theorem}  \label{mixed_reg}
If $\breve{\Omega}$ is convex, then the solution to (\ref{mixed1}) denoted by $(\bz,p)\in \bH_r(\dive^k,\Omega)\times L^2_r(\Omega)$ satisfies
\begin{align*}
 \| p \|_{H^2_r(\Omega)}  +   \| p \|_{\tilde{H}^1_r(\Omega)} + \| \partial_r p \|_{L^2_{1/r}(\Omega)} &\leq C\| f \|_{L^2_r(\Omega)}, \\
 \| \bz_{rz} \|_{H^1_r(\Omega)} + \| kz_\theta-z_r \|_{\tilde{H}_r^1(\Omega)} &\leq C\| f \|_{L^2_r(\Omega)} 
\end{align*}
for all $|k| \geq 1$.
\end{theorem}
\begin{proof}
If $\breve{\Omega}$ is convex, by \cite[page 589]{FSCM_Part2}, the solution $p$ to (\ref{cts1}) is in
$H^2_{(k)}(\Omega)\cap H_{r,0}^1(\Omega)$,
where 
\begin{align*}
H^2_{(\pm 1)}(\Omega) &= \{w\in H^2_r(\Omega): w=0 \tn{ on } \Gamma_0 \}, \\
H^2_{(k)}(\Omega) &= H_r^2(\Omega) \cap H^1_{1/r}(\Omega) \tn{ for } |k|\geq 2.
\end{align*}
Then the first regularity result follows by \cite[Theorem 3.2]{FSCM_Part2} for $|k|\geq 2$
and by  \cite[page 594]{FSCM_Part2} for $|k|=1$.

Next, notice that $\bz=-\grad\nolimits_{rz}^{k*}p$, and consider the complex vector-valued function $\bq=(z_r,  iz_\theta, z_z)^T$. 
Then, by direct calculation, one can show that
\begin{equation} \label{z_est}
\begin{aligned}
\bcurl\nolimits_k \bq &= 0, \\
\dive\nolimits_k \bq &= f, \\
\bq_{rz} \cdot \bt &= 0 \tn{ on } \Gamma_1, \\
q_\theta &= 0 \tn{ on } \Gamma_1,
\end{aligned}
\end{equation}
where 
\begin{align*}
\dive\nolimits_k\bw &= \dfrac{1}{r}\partial_r(rw_r) + \dfrac{ik}{r}w_\theta + \partial_zw_z, \\
(\bcurl\nolimits_k\bw)_r &= \dfrac{ik}{r}w_z-\partial_zw_\theta, \\ 
(\bcurl\nolimits_k\bw)_\theta &= \partial_zw_r-\partial_rw_z, \\
(\bcurl\nolimits_k\bw)_z &= \dfrac{1}{r}(\partial_r(rw_\theta)-ikw_r).
\end{align*}
The boundary conditions in (\ref{z_est}) are true since $p=0$ on $\Gamma_1$. 
Therefore, 
\[
\bq \in \bX_{(k)}(\Omega) := \{\bv\in \bL^2_r(\Omega): \bcurl\nolimits_k\bv \in \bL^2_r(\Omega) \tn{ and } \dive\nolimits_k\bv\in L^2_r(\Omega) \tn{ and } \bv_{rz}\cdot \bt = 0 \tn{ and } v_\theta=0 \tn{ on } \Gamma_1\}.
\]
By \cite[Theorem 2.10]{FSCM_Maxwell}, this space is continuously embedded in $\bH^1_{(k)}(\Omega)$ where
\begin{align*}
\bH^1_{(\pm 1)}(\Omega)&=\{\bw\in H_r^1(\Omega)\times H_r^1(\Omega) \times \tilde{H}_r^1(\Omega): w_r\pm iw_\theta\in L^2_{1/r}(\Omega) \}, \\
\bH^1_{(k)}(\Omega) &= \tilde{H}_r^1(\Omega) \times \tilde{H}_r^1(\Omega) \times \tilde{H}_r^1(\Omega) \tn{ for } |k|\geq 2.
\end{align*}
For all $|k|\geq 1$, this continuous embedding proves the second result.

\end{proof}

Next, we consider the following discrete version of (\ref{mixed1}):
\begin{equation} \label{mixed_disc}
\begin{aligned}
(\bz_h, \bw_h)_r - (p_h, \dive\nolimits_{rz}^k \bw_h)_r &= 0 &&\tn{ for all } \bw_h\in\bC_h, \\
(\dive\nolimits_{rz}^k\bz_h, s_h)_r &= (f,s_h)_r &&\tn{ for all } s_h\in D_h.
\end{aligned}
\end{equation}
Stability and convergence results of (\ref{mixed_disc}) were proved in 
\cite[Theorem 4.1]{O:2019}. We prove a more concrete error estimate here. 

\begin{theorem} \label{error}
Suppose $(\bz,p)\in\bH_r(\dive^k,\Omega)\times L^2_r(\Omega)$ solve (\ref{mixed1}) and
$(\bz_h,p_h)\in \bC_h\times D_h$ solve $(\ref{mixed_disc})$. If $f\in D_h$, we have the following error estimates for all $|k|\geq 1$:
\begin{align*}
 \|\bz-\bz_h\|_{L^2_r(\Omega)} &\leq Ch\|f\|_{L^2_r(\Omega)}, \\
 \|\Pi_h^S p - p_h\|_{L^2_r(\Omega)} &\leq Ch^2\|f\|_{L^2_r(\Omega)},
\end{align*}
where $\Pi_h^S: L^2_r(\Omega) \rightarrow D_h$ denotes the $L^2_r$-orthogonal projection onto $D_h$.
\end{theorem}

\begin{proof}
Throughout this proof, we assume that $f\in D_h$.
Let $\bw_h=\bz_h-\Pi_h^{d,k}\bz$. Then,  since $\dive\nolimits_{rz}^k\bz=f \in D_h$, 
\begin{equation} \label{boo}
\dive\nolimits_{rz}^k\bz = \dive\nolimits_{rz}^k\bz_h, 
\end{equation}
and so
\begin{equation} \label{boo2}
\begin{aligned}
\dive\nolimits_{rz}^k\bw_h &= \dive\nolimits_{rz}^k(\bz_h-\Pi_h^{d,k}\bz), \\
&= \dive\nolimits_{rz}^k\bz_h - \Pi_h^{o,k} \dive\nolimits_{rz}^k\bz &&\tn{ by  Theorem~\ref{ProjSet}~item~(\ref{comm})}, \\
&= 0 &&\tn{ by (\ref{boo}).}
\end{aligned}
\end{equation}
We also have that 
\begin{equation} \label{boo3}
\begin{aligned}
(\bz-\bz_h,\bv_h)_r - (p-p_h, \dive\nolimits_{rz}^k\bv_h)_r &= 0 &&\tn{ for all } \bv_h\in\bC_h, \\
(\dive\nolimits_{rz}\bz-\dive\nolimits_{rz}^k\bz_h, s_h)_r &= 0 &&\tn{ for all } s_h\in D_h
\end{aligned}
\end{equation}
Therefore, by (\ref{boo2}) and (\ref{boo3}), it follows that
\begin{equation*}
\begin{aligned}
(\bz-\bz_h, \bw_h) &= 0, \\
(\bz-\bz_h, \bz-\Pi_h^{d,k}\bz)_r &= (\bz-\bz_h, \bz-\bz_h)_r. 
\end{aligned}
\end{equation*}
Therefore,
\[
\|\bz-\bz_h\|_{\Lr} \leq \|\bz-\Pi_h^{d,k}\bz\|_{\Lr}.
\]
This together with Theorem~\ref{concreteD} and Theorem~\ref{mixed_reg}  completes the proof of the first estimate of the Theorem. 

To prove the second estimate of the Theorem, we first let $(\epsilon_{\bz}, \epsilon_{p})$ and $(\epsilon_{\bz,h}, \epsilon_{p,h})$ be the solution to (\ref{mixed1}) and (\ref{mixed_disc}) respectively
with $f$ replaced with $\Pi_h^S p - p_h\in D_h$. Then,
\begin{align*}
\|\Pi_h^S p - p_h\|_{\Lr}^2 &= (\dive\nolimits_{rz}^k\epsilon_{\bz,h}, \Pi_h^S p - p_h)_r  &&\tn{ by definition of } \epsilon_{\bz,h}, \\
&= (\dive\nolimits_{rz}^k\epsilon_{\bz,h}, p - p_h)_r, \\
&= (\bz-\bz_h, \epsilon_{\bz,h})_r &&\tn{ by (\ref{boo3}), } \\ 
&= (\bz-\bz_h, \epsilon_{\bz,h}-\epsilon_{\bz})_r &&\tn{ by (\ref{boo}). }
\end{align*}
Therefore,
\[
\|\Pi_h^S p - p_h\|_{\Lr}^2 \leq \| \bz-\bz_h \|_{\Lr} \| \epsilon_{\bz}-\epsilon_{\bz,h}\|_{\Lr} \leq Ch^2\| f \|_{\Lr}\|\Pi_h^S p - p_h \|_{\Lr}, 
\]
where in the last inequality, we are using the first estimate of the Theorem twice.
This completes the proof of the second estimate of the Theorem.
\end{proof}

\section{The Multigrid Algorithm} \label{multigrid}

We consider a sequence of nested meshes $\TT=\{\TT_1, \TT_2, \cdot\cdot\cdot, \TT_L \}$ for the multigrid algorithm.
In particular, $\TT_1$ is the coarsest level mesh, and $\TT_l$ is obtained by connecting the midpoints of all edges in $\TT_{l-1}$ for $l=2, 3, \cdot\cdot\cdot, L$. \
Throughout this paper, we assume that $\TT$ satisfy this property.
Let $\bC_l$ denote the discrete space $\bC_h$ on the $l$-th level mesh. Define $\Lambda_l: \bC_l \rightarrow \bC_l$ in the following way:
\[
(\Lambda_l^k\bv_l,\bw_l) = \Lambda^k(\bv_l, \bw_l) \tn{ for all } \bv_l, \bw_l \in \bC_l.
\]
Since we are assuming that $|k|\geq 1$ is fixed, we   
write $\Lambda_l$ instead of $\Lambda_l^k$ for simplicity of notation as we are doing for $\Lambda^k$.
In order to approximate the solution $\bu\in \bH_{r}(\dive\nolimits^k,\Omega)$ that satisfies
\[
\Lambda(\bu,\bv) = (\bff, \bv)_r \tn{ for all } \bv\in \bH_{r}(\dive\nolimits^k,\Omega),
\]
the multigrid algorithm presented here will solve 
\[
\Lambda_L \bu_L =\bff_L
\]
on the finest level mesh by using the sequence of meshes $\TT$. The right-hand-side function $\bff_L$ denotes the usual representation of the data function $\bf$. 
Multigrid will use the sequence of meshes $\TT$ to provide a solution in $\bC_L$.

For the multiplicative subspace correction method, we will use the following subspace decomposition of $\bC_l$ as in \cite{AFW:2000}:
\begin{equation} \label{decomp}
\bC_l = \sum_{\nu\in\mathcal{V}_l } \bC_l^\nu,
\end{equation}
where $\mathcal{V}_l$ denotes the set of mesh vertices in the $l$-th level mesh, $D_\nu$ denotes the vertex patch of $\nu\in\mathcal{V}_l$ (the union of all triangles that have $\nu$ as a vertex), and 
\[
\bC_l^\nu = \{\bw_l\in \bC_l: \mathrm{supp}(\bw_l)\subset D_\nu \}.
\]
We will use the decomposition (\ref{decomp}) to construct additive and multiplicative subspace correction methods.
We will present here the block Gauss-Seidel type multiplicative smoothing iteration $\bu_{i+1}=\mathrm{gs}(\bu_i,\bff)$. 
Let $\bC_{l,j}$ for $j=1,2,\cdot\cdot\cdot, N_l$ be the enumeration of 
subspaces appearing in the subspace decomposition (\ref{decomp}) where $N_l$ is the number of vertices in $\TT_l$. 
Then, define $\Lambda_{l,j}: \bC_{l,j}\rightarrow \bC_{l,j}$ as
\[
(\Lambda_{l,j}\bv,\bw) = \Lambda(\bv, \bw) \tn{ for all } \bv, \bw\in \bC_{l,j},
\]
and $Q_{l,j}$ as the $L^2_r$-orthogonal projection onto $ \bC_{l,j}$. Similarly, we will use $Q_l$ to denote the $L^2_r$-orthogonal projection onto $\bC_l$.
\begin{algorithm}(multiplicative smoothing) Given $\bu_i\in \bC_l$, $\bu_{i+1}=\mathrm{gs}(\bu_i,\bff)$ in $\bC_l$ is computed in the following way:
\begin{enumerate}
\item Set $\bu_i^{(0)}=\bu_i$.
\item For $j=1,2,\cdot\cdot\cdot, N_l$, compute
\[
\bu_i^{(j)} = \bu_i^{(j-1)} + \Lambda_{l,j}^{-1}Q_{l,j}(\bff-\Lambda_l\bu_i^{(j-1)}).
\]
\item Set $\bu_{i+1}=\bu_i^{(N_l)}$.
\end{enumerate}
\end{algorithm}
Standard arguments show that  $\bu_{i+1}=\mathrm{gs}(\bu_i,\bff)$ can be rewritten as
\[
\bu_{i+1}=\bu_i + R_l(\bff-\Lambda_l\bu_i), 
\]
where
\[
R_l = (I-(I-P_{l,N_l})(I-P_{l,N_{l,N_l-1}})\cdot\cdot\cdot(I-P_{l,1}))\Lambda_l^{-1},
\]
where $P_{l,j}$ denotes the orthogonal projection onto $\bC_{l,j}$ with respect to the $\Lambda(\cdot,\cdot)$-inner product.
Now we are ready to state the multigrid algorithm. 
\begin{algorithm}(Multigrid V-cycle) Given $\bu$ and $\bff$ in $\bC_l$, define the output $\mathrm{mg}_l(\bu,\bff)$ in $\bC_l$ by the following recursive procedure:
\begin{enumerate}
\item Set $\mathrm{mg}_1(\bu,\bff)=\Lambda_1^{-1}\bff$.
\item For $l>1$, define $\mathrm{mg}_l(\bu,\bff)$ recursively:
\begin{enumerate}
\item (pre-smoothing) $\bv^{(1)}=\bu+R_l(\bff-\Lambda_l\bu)$.
\item (coarse grid correction) $\bv^{(2)}=\bv^{(1)} + \mathrm{mg}_{l-1}(\mathbf{0},Q_{l-1}(\bff-\Lambda_l\bv^{(1)}))$.
\item (post-smoothing) $\bv^{(3)} = \bv^{(2)}+R_l^t(\bff-\Lambda_l\bv^{(2)})$.
\item $\mathrm{mg}_l(\bu,\bff)=\bv^{(3)}$.
\end{enumerate}
\end{enumerate}
\end{algorithm}
The V-cycle iterates $\bx_{i+1}=\mathrm{mg}_l(\bx_i,\bff)$, approximating $\bx=\Lambda_l^{-1}\bff$, are connected through a linear error reduction operator $\mathit{E}_l$, i.e.,
\begin{equation} \label{e_reduction}
\bx_{i+1}-\bx = \mathit{E}_l(\bx_i-\bx).
\end{equation}
(See \cite{mg:book} for details.) The following theorem is the main result of this paper. 
\begin{theorem} \label{main}
Assume that $\breve{\Omega}\subset \RRR^3$ is convex. Then, for any fixed integer $|k|\geq 1$, there exists $0<\delta<1$ independent of the meshsize and refinement level that satisfies 
\[
0 \leq \Lambda(\mathit{E}_l\bu,\bu) \leq \delta\Lambda(\bu,\bu) \tn{ for all } \bu\in\bC_l \tn{ and all } l\geq1.
\]
\end{theorem}
This theorem together with (\ref{e_reduction}) implies that the multigrid V-cycle will converge to the exact solution and that the convergence rate is independent of the meshsize $h$.
The proof of Theorem~\ref{main} is given at the end of the next section after the necessary tools are developed. 

\section{Multigrid Analysis} \label{multigrid2}

This section is devoted to proving Theorem~\ref{main}. Following the standard abstract framework for multigrid analysis~\cite{AFW:2000,BH:1983,mg:book,mg:handbook},
we will verify the two conditions stated in the following lemma. 
We closely follow the steps introduced in \cite{AFW:2000} to accomplish this task. 
All results proved in this section holds true for any fixed integer $|k|\geq 1$, and $P_l$ denote the orthogonal projection onto $\bC_l$ with respect to the $\Lambda(\cdot,\cdot)$-inner product. 

\begin{lemma}        \label{lem:cond}
Theorem~\ref{main} follows from the
two conditions below:
\begin{enumerate}
\item {\em Existence of a stable decomposition:} There exists a constant
  $C_1 > 0$ independent of the meshsizes and $l$, such that for all $\bv$ 
  in $(I - P_{l-1})\bC_l$, there is a decomposition 
  \[
  \bv = \sum_{j=1}^{N_{l}}\bv_{j}
  \tn{ with } \bv_{j} \in\bC_{l,j},
  \]
  satisfying
  \[
  \sum_{j=1}^{N_{l}} \Lambda(\bv_{j}, \bv_{j}) \leq C_1 \Lambda(\bv, \bv).
  \]

  \item {\em Limited interaction:} There exists a constant $C_2 > 0$,
    independent of $l$, such that
    \[
    \sum_{j=1}^{N_{l}}\sum_{m=1}^{N_{l}}
    |\Lambda(\bv_{j}, \bw_{m})| \leq 
    C_2 
    \left(\sum_{j=1}^{N_{l}}\Lambda(\bv_{j}, \bv_{j})\right)^{\frac{1}{2}}
    \left(\sum_{m=1}^{N_{l}}\Lambda(\bw_{m}, \bw_{m})\right)^{\frac{1}{2}}
    \]
    for all $\bv_{j}\in\bC_{l,j}$, $\bw_{m}\in\bC_{l,m}$, and $l \geq1$. 
\end{enumerate}
\end{lemma}
The two conditions in Lemma~\ref{lem:cond} are proved at the end of this section.

It was shown in \cite{O:2015} that the following family of Fourier finite element spaces form an exact sequence:
\begin{equation} \label{exactt}
0 \rightarrow A_h \rightarrow \bB_h \rightarrow \bC_h \rightarrow D_h \rightarrow 0.
\end{equation}

Let $\grad\nolimits^k_h: D_h \rightarrow \bC_h$ denote the $L_r^2$-adjoint of the map of $-\dive\nolimits_{rz}^k: \bC_h\rightarrow D_h$, i.e.,
\[
(-\dive\nolimits_{rz}\bv_h, s_h)_r = (\bv_h, \grad\nolimits^k_hs_h)_r \quad\tn{ for all } \bv_h\in \bC_h \tn{ and } s_h\in D_h.
\]
Then, due to the exactness of (\ref{exactt}),
we have the following discrete Helmholtz decomposition \cite{O:2015} and \cite[page 10]{O:2019}:
\begin{equation} \label{disc_helm}
\bC_h = \curl_{rz}^k\bB_h + \grad\nolimits^k_hD_h.
\end{equation}

When discussing two consecutive level meshes and the corresponding finite element spaces $\bC_{l-1}$ and $\bC_{l}$, for example, we will use $H$ to denote the meshsize on $\bC_{l-1}$
and $h$ to denote the meshsize on $\bC_{l}$. Since we obtain the next level mesh by connecting the midpoints of all edges in the previous level mesh,
we have that $H=2h$. $\bC_H$ and $\bC_h$ will often be used in place of $\bC_{l-1}$ and $\bC_{l}$ respectively. $D_H$ and $D_h$ are connected in a similar way.  
Recall that $\Pi_h^S$ is the $L^2_r$-orthogonal projection on to $D_h$, and similarly $\Pi_H^S$ is the $L^2_r$-orthogonal projection on to $D_H$.
With this notation, we have the following lemma.

\begin{lemma} \label{lem1}
For all $p_h\in D_h$, we have 
\[
\|p_h-\Pi_H^Sp_h\|_{\Lr} \leq CH\| \grad\nolimits^k_hp_h \|_{\Lr}
\]
\end{lemma}

\begin{proof}
Given $p_h\in D_h$, define $\bz\in \bH_{r}(\dive\nolimits^k,\Omega)$ and $p\in L^2_r(\Omega)$ as the solution of (\ref{mixed1}) with $f=-\dive\nolimits_{rz}^k \grad\nolimits^k_hp_h\in D_h$.
Then, $\bz_h= -\grad\nolimits^k_hp_h\in \bC_h$ and $p_h\in D_h$ will be the solution of (\ref{mixed_disc}) with the same $f$. 
By triangle inequality, we have
\begin{equation} \label{lemboo0}
\|p_h-\Pi_H^Sp_h\|_{\Lr} \leq \|p_h-p\|_{\Lr} + \|p-\Pi_H^Sp\|_{\Lr} + \|\Pi_H^Sp-\Pi_H^Sp_h\|_{\Lr}. 
\end{equation}
Let us first prove that 
\begin{equation} \label{lemboo1}
\|p-\Pi_H^Sp\|_{\Lr} \leq CH\| \grad\nolimits^k_hp_h \|_{\Lr}.
\end{equation}
It is known in the literature \cite[Lemma 5]{BBD:2006} that 
\begin{equation} \label{lemboo2}
\|p-\Pi_H^Sp\|_{\Lr} \leq CH|p|_{H^1_r(\Omega)}.
\end{equation}
By definition of $\grad\nolimits_{rz}^{k*}$, it is clear that 
\begin{equation} \label{lemboo3}
|p|_{H^1_r(\Omega)} \leq \|\grad\nolimits_{rz}^{k*}p\|_{\Lr}.
\end{equation}
Furthermore, we have 
\begin{align*}
 \|\grad\nolimits_{rz}^{k*}p\|_{\Lr}^2 &= (\grad\nolimits_{rz}^{k*}p, -\bz)_r, \\
 &= (p, \dive\nolimits_{rz}^k\bz)_r, \\
 &= (p, \dive\nolimits_{rz}^k\bz_h)_r &&\tn{ since } f\in D_h, \\
 &= (\grad\nolimits_{rz}^{k*}p, -\bz_h)_r, \\
 &=  (\grad\nolimits_{rz}^{k*}p, \grad\nolimits^k_hp_h)_r.
\end{align*}
Therefore, 
\begin{equation} \label{lemboo4}
 \|\grad\nolimits_{rz}^{k*}p\|_{\Lr} \leq \|  \grad\nolimits^k_hp_h \|_{\Lr},
\end{equation}
and thus (\ref{lemboo1}) is proved by (\ref{lemboo2})--(\ref{lemboo4}).

To bound the first term appearing on the right hand side of (\ref{lemboo0}),
we will next show that
\begin{equation} \label{lemboo5}
 \|p_h-p\|_{\Lr} \leq Ch\| \grad\nolimits^k_hp_h \|_{\Lr}.
\end{equation}
Since orthogonal projections have unit norm, 
(\ref{lemboo5}) will also prove that
\begin{equation} \label{lemboo6}
 \|\Pi_H^Sp-\Pi_H^Sp_h\|_{\Lr} \leq Ch\| \grad\nolimits^k_hp_h \|_{\Lr}.
\end{equation}

First of all, by Theorem~\ref{error}, we have that
\begin{equation} \label{lemboo7}
\|p_h-\Pi_h^Sp\|_{\Lr} \leq Ch^2\|\dive\nolimits_{rz}^k\grad\nolimits^k_h p_h \|_{\Lr} \leq Ch\|\grad\nolimits^k_h p_h \|_{\Lr}.
\end{equation}
The second inequality above follows from known inverse inequalities in weighted spaces \cite[Proposition 2.1]{GO:2012} and \cite[Proposition 5.1]{O:2019}.
Therefore, by (\ref{lemboo7}) and (\ref{lemboo1}), we have
\[
 \|p_h-p\|_{\Lr} \leq  \|p_h-\Pi_h^Sp\|_{\Lr} + \|\Pi_h^Sp-p\|_{\Lr} \leq Ch\|\grad\nolimits^k_h p_h \|_{\Lr}. 
\]
Since $H\leq Ch$, the proof is complete by
 (\ref{lemboo0}), (\ref{lemboo1}), (\ref{lemboo5}), and (\ref{lemboo6}).
\end{proof}

Let $P_H: \bC_h \rightarrow \bC_H$ denote the $\Lambda$-orthogonal projection. The following lemma is crucial in proving the uniform convergence of the multigrid algorithm.
\begin{lemma} \label{lem2}
Let $\bw_h\in \bC_h$. Let 
\[
\bw_h - P_H\bw_h = \curl_{rz}^k\bb_h + \grad\nolimits^k_hd_h
\]
be the discrete Helmholtz decomposition of $\bw_h - P_H\bw_h$ with $\bb_h\in\bB_h$ and $d_h\in D_h$.
Then
\begin{align*}
\| \bb_h \|_{\Lr} &\leq CH\| \bw_h - P_H\bw_h \|_{\Lr}, \\
\| \grad\nolimits^k_hd_h \|_{\Lr} &\leq CH\| \bw_h - P_H\bw_h \|_{\Lambda}.
\end{align*}
\end{lemma}

\begin{proof}
Define $\bz_h\in \bC_h$ to be the solution to 
\[
\Lambda(\bz_h, \bq_h) = (\grad\nolimits^k_hd_h, \bq_h)_r \tn{ for all } \bq_h\in \bC_h.
\]
It follows that
\begin{equation} \label{ggg}
(\bz_h,\curl_{rz}^k\bv_h) _{L^2_r(\Omega)} = \Lambda(\bz_h,\curl_{rz}^k\bv_h) = 0 \quad\tn{ for all } \bv_h\in \bB_h,
\end{equation}
so $\bz_h=-\grad\nolimits^k_hp_h$ for some $p_h\in D_h$ by (\ref{disc_helm}).
Then
\begin{center}
$(\bz_h,p_h)\in \bC_h\times D_h$ is the solution of (\ref{mixed_disc}) with $f=\dive\nolimits_{rz}^k\bz_h\in D_h$.
\end{center}
Now, let
\begin{center} 
$(\bz,p)\in \hdivk \times \Lr$ be the solution of (\ref{mixed1}) with $f=\dive\nolimits_{rz}^k\bz_h\in D_h$,
\end{center}
and
\begin{center}
$(\bz_H,p_H)\in \bC_H\times D_H$ be the solution of (\ref{mixed_disc}) with $f=\dive\nolimits_{rz}^k\bz_h\in D_h$.
\end{center}
Note that 
\begin{equation} \label{f1}
\dive\nolimits_{rz}^k\bz=\dive\nolimits_{rz}^k\bz_h,
\end{equation}
and 
\begin{equation} \label{f3}
\dive\nolimits_{rz}^k\bz_H=\Pi_H^S\dive\nolimits_{rz}^k\bz_h.
\end{equation}
We further define 
\begin{center} 
$(\tilde{\bz},\tilde{p})\in \hdivk \times \Lr$ to be the solution of (\ref{mixed1}) with $f=\dive\nolimits_{rz}^k\bz_H\in D_H$.
\end{center}
In this case, since $f\in D_H$,
\begin{equation} \label{f2}
\dive\nolimits_{rz}^k\tilde{\bz}=\dive\nolimits_{rz}^k\bz_H.
\end{equation}
Then, 
\begin{equation} \label{lemb0}
\begin{aligned}
\| \grad\nolimits^k_hd_h \|_{\Lr}^2 &= \Lambda(\bz_h, \grad\nolimits^k_hd_h), \\
&= \Lambda(\bz_h, \bw_h - P_H\bw_h) &&\tn{ by (\ref{ggg}), }  \\
&=  \Lambda(\bz_h-\bz_H, \bw_h - P_H\bw_h), \\
&\leq \|\bz_h-\bz_H \|_{\Lambda} \|\bw_h - P_H\bw_h \|_{\Lambda}. 
\end{aligned}
\end{equation}
The second inequality of the lemma is proved once we show that
\[
\|\bz_h-\bz_H \|_{\Lambda} \leq CH\| \grad\nolimits^k_hd_h \|_{\Lr}.
\]
To do this, let us first consider $\|\bz-\tilde{\bz} \|_{\Lr}$:
\begin{align*}
\|\bz-\tilde{\bz}\|_{\Lr}^2 &= (\bz-\tilde{\bz},\bz-\tilde{\bz})_r, \\
&= -(\grad\nolimits_{rz}^{k*}p - \grad\nolimits_{rz}^{k*}\tilde{p}, \bz-\tilde{\bz})_r, \\
&= (p-\tilde{p}, \dive\nolimits_{rz}^{k}\bz-\dive\nolimits_{rz}^{k}\tilde{\bz}), \\
&= (p-\tilde{p}, \dive\nolimits_{rz}^{k}\bz_h-\dive\nolimits_{rz}^{k}\bz_H) &&\tn{ by (\ref{f1}) and (\ref{f2}), } \\
&= (p-\tilde{p}, \dive\nolimits_{rz}^{k}\bz_h-\Pi_H^S\dive\nolimits_{rz}^{k}\bz_h) &&\tn{ by (\ref{f3}), } \\
&= ((p-\tilde{p}) - \Pi_H^S(p-\tilde{p}), \dive\nolimits_{rz}^{k}\bz_h-\Pi_H^S\dive\nolimits_{rz}^{k}\bz_h), \\ 
&\leq CH|p-\tilde{p}|_{H^1_r(\Omega)}\|\dive\nolimits_{rz}^{k}\bz_h-\Pi_H^S\dive\nolimits_{rz}^{k}\bz_h\|_{\Lr},  \\ 
&\leq CH|p-\tilde{p}|_{H^1_r(\Omega)}\|\dive\nolimits_{rz}^{k}\bz_h\|_{\Lr},  \\ 
&\leq CH\|\bz-\tilde{\bz}\|_{\Lr}\|\dive\nolimits_{rz}^{k}\bz_h\|_{\Lr} &&\tn{ since } \bz-\tilde{\bz}=-\grad\nolimits_{rz}^{k*}(p-\tilde{p}).
\end{align*}
Therefore, we have
\begin{equation}  \label{f4}
\|\bz-\tilde{\bz}\|_{\Lr} \leq CH\|\dive\nolimits_{rz}^{k}\bz_h\|_{\Lr},
\end{equation}
and so we have
\begin{equation} \label{f5}
\begin{aligned}
\|\bz_h-\bz_H\|_{\Lr} &\leq \|\bz_h-\bz\|_{\Lr} + \|\bz-\tilde{\bz}\|_{\Lr}+\|\tilde{\bz}-\bz_H\|_{\Lr}, \\ 
&\leq Ch\|\dive\nolimits_{rz}^{k}\bz_h\|_{\Lr} + CH\|\dive\nolimits_{rz}^{k}\bz_h\|_{\Lr} + CH\|\dive\nolimits_{rz}^{k}\bz_H\|_{\Lr} \\ &\quad\quad\tn{ by Theorem~\ref{error} and (\ref{f4}), } \\
&\leq CH\|\dive\nolimits_{rz}^{k}\bz_h\|_{\Lr} &&\tn{ by (\ref{f3}). }
\end{aligned}
\end{equation}
Hence,
\begin{equation} \label{f8}
\begin{aligned}
\|\bz_h-\bz_H\|_\Lambda^2 &=  \|\bz_h-\bz_H\|^2_{\Lr} +  \|\dive\nolimits_{rz}^{k}(\bz_h-\bz_H)\|^2_{\Lr}, \\
&\leq CH^2\|\dive\nolimits_{rz}^{k}\bz_h\|_{\Lr}^2 + \|\dive\nolimits_{rz}^{k}\bz_h - \Pi_H^S \dive\nolimits_{rz}^{k}\bz_h\|_{\Lr}^2 &&\tn{ by (\ref{f5}) and (\ref{f3})}, \\
&\leq CH^2(\|\dive\nolimits_{rz}^{k}\bz_h\|_{\Lr}^2 + \| \grad\nolimits^k_h\dive\nolimits_{rz}^{k}\bz_h\|^2_{\Lr}) &&\tn{ by Lemma~\ref{lem1}, } \\
&\leq CH^2\|\grad\nolimits^k_hd_h\|_{\Lr}^2.
\end{aligned}
\end{equation}
The last inequality above holds, since
\[
\|\dive\nolimits_{rz}^{k}\bz_h\|_{\Lr}^2 + \| \grad\nolimits^k_h\dive\nolimits_{rz}^{k}\bz_h\|^2_{\Lr} \leq \|\Lambda_h\bz_h\|_{\Lr}^2 \leq \|\grad\nolimits^k_hd_h\|_{\Lr}^2.
\]
We then reach the second inequality of the Lemma by (\ref{f8}) and (\ref{lemb0}).

Now let us prove the first inequality of the Lemma. To do so, let $\bF=\bcurl\nolimits^k_{rz}\bb_h$ and consider the following boundary value problem:
\begin{equation} \label{bdy}
\begin{aligned}
\bb &=\bcurl\nolimits^{k*}_{rz}\br, \\
\bcurl\nolimits^k_{rz}\bb &= \bF, \\
\dive\nolimits_{rz}^k\br &=0, \\
\br_{rz} \cdot \bt &= 0 \tn{ and } r_\theta = 0 &&\tn{ on $\Gamma_1$. }
\end{aligned}
\end{equation}
This is a subproblem of one of the weighted Hodge Laplacian problems studied in \cite{O:2019}. Such subproblems were called $\mathfrak{B}^k$ problems in \cite{AFW:2010}.

By the definition of $\bb$ and $\bb_h$ and Theorem~\ref{ProjSet} item (\ref{comm}), we have 
\[
(\bb-\bb_h,\bb_h-\Pi_h^{c,k}\bb)_r=0.
\]
This implies that
\[
(\bb-\bb_h,\bb-\bb_h)_r = (\bb-\bb_h, \bb-\Pi_h^{c,k}\bb)_r,
\] 
and so
\begin{equation} \label{a1}
\begin{aligned}
\|\bb-\bb_h\|_{\Lr} &\leq \|\bb-\Pi^{c,k}_h\bb\|_{\Lr}, \\
&\leq Ch(|b_\theta|_{H^1_r(\Omega)} + \|kb_z\|_{\tilde{H}^1_r(\Omega)}  +  \|kb_r+b_\theta \|_{\tilde{H}^1_r(\Omega)}) &&\tn{ by Theorem~\ref{cc}, } \\
&\leq Ch\|\bcurl\nolimits^k_{rz}\bb\|_{\Lr} = Ch\|\bcurl\nolimits^k_{rz}\bb_h\|_{\Lr}. 
\end{aligned}
\end{equation}
The last inequality above follows from the continuous embedding result in \cite[Theorem 2.10]{FSCM_Maxwell}.
Next, notice that
\begin{equation} \label{a0}
(\bcurl\nolimits_{rz}^{k}\bb_h, \bcurl\nolimits_{rz}^{k}\bv_H)_r = \Lambda(\bw_h-P_H\bw_h, \bcurl\nolimits_{rz}^{k}\bv_H)=0 \quad\tn{ for all } \bv_H\in \bB_H.
\end{equation}
Then, since $\br\in \mathrm{null}(\dive\nolimits_{rz}^k)=\mathrm{range}(\bcurl\nolimits_{rz}^{k})$ and $\Pi_H^{d,k}$ satisfies the commuting diagram property,  we have
\begin{equation} \label{a2}
\begin{aligned}
\| \bb \|_{\Lr}^2 &= (\bb,\bcurl\nolimits_{rz}^{k*}\br)_r, \\
&= (\bcurl\nolimits_{rz}^{k}\bb, \br)_r, \\
&= (\bcurl\nolimits_{rz}^{k}\bb_h, \br)_r, \\ 
&=  (\bcurl\nolimits_{rz}^{k}\bb_h, \br-\Pi_H^{d,k}\br)_r &&\tn{ by~(\ref{a0}), } \\
&\leq \|\bcurl\nolimits_{rz}^{k}\bb_h\|_{\Lr}\| \br-\Pi_H^{d,k}\br \|_{\Lr}, \\
&\leq  CH\|\bcurl\nolimits_{rz}^{k}\bb_h\|_{\Lr}(|\br_{rz}|^2_{H^1_r(\Omega)}+ \| kr_\theta -r_r \|_{\tilde{H}^1_r(\Omega)}) &&\tn{ by Theorem~\ref{concreteD}, }\\
&\leq CH\|\bcurl\nolimits_{rz}^{k}\bb_h\|_{\Lr}\|\bcurl\nolimits_{rz}^{k*}\br\|_{\Lr} &&\tn{ by  \cite[Theorem 2.10]{FSCM_Maxwell},} \\
\end{aligned}
\end{equation}
Since $\bb=\bcurl\nolimits_{rz}^{k*}\br$, (\ref{a2}) implies that
\begin{equation} \label{a3}
\| \bb \|_{\Lr} \leq CH\|\bcurl\nolimits_{rz}^{k}\bb_h\|_{\Lr}.
\end{equation}
Therefore, by the triangle inequality, (\ref{a1}), and (\ref{a3}), we conclude that
\begin{equation} \label{a4}
\| \bb_h \|_{\Lr} \leq CH\|\bcurl\nolimits_{rz}^{k}\bb_h\|_{\Lr} \leq CH\| \bw_h - P_H\bw_h \|_{\Lr}, 
\end{equation}
and this completes the proof.
\end{proof}

Now we are ready to prove Theorem~\ref{main}. 

\vspace{\baselineskip}
\underline{\itshape Proof of Theorem~\ref{main}.  }

\vspace{\baselineskip}
The proof will be complete once we verify the two conditions stated in Lemma~\ref{lem:cond}. 
The proof of the limited interaction is standard \cite{AFW:2000, mg:handbook}, so we will only prove the existence of a stable decomposition.
In other words, we will show that the subspace decomposition (\ref{decomp}) is stable in the sense that if $\bv_l\in(I - P_{l-1})\bC_l$ and
\[
\bv_l = \sum_{j=1}^{N_{l}}\bv_{l,j} \tn{ where } \bv_{l,j}\in \bC_{l,j}
\]
is the decomposition (\ref{decomp}), then
  \[
  \sum_{j=1}^{N_{l}} \Lambda(\bv_{l,j}, \bv_{l,j}) \leq C_1 \Lambda(\bv_l, \bv_l).
  \]
Consider the following Helmholtz decomposition:
 \begin{equation} \label{h1}
 \bv_l=\bcurl\nolimits_{rz}^k\bb_l + \bbc_l.
 \end{equation} 
Then, by (\ref{decomp}), 
\[
\bbc_l = \sum_{j=1}^{N_l} \bbc_{l,j} \tn{ for } \bbc_{l,j}\in \bC_{l,j},
\]
and by considering a decomposition like (\ref{decomp}) for $\bB_l$, we also have
\[
\bb_l=\sum_{j=1}^{N_l}\bb_{l,j} \tn{ for } \bb_{l,j}\in\bB_{l,j}.
\]
Then, $\bv_{l,j}=\bcurl\nolimits_{rz}^k\bb_{l,j} + \bbc_{l,j}$, and 
\begin{align*}
\sum_{j=1}^{N_{l}} \| \bv_{l,j} \|_\Lambda^2 &= \sum_{j=1}^{N_{l}}  \| \bcurl\nolimits_{rz}^k\bb_{l,j} + \bbc_{l,j} \|_\Lambda^2, \\
&= \sum_{j=1}^{N_{l}} \| \bcurl\nolimits_{rz}^k\bb_{l,j}  \|_{\Lr}^2 + \| \bbc_{l,j} \|_{\Lr}^2 + \| \dive\nolimits_{rz}^k\bbc_{l,j} \|_{\Lr}^2, \\
&\leq  C\sum_{j=1}^{N_{l}} h^{-2} \| \bb_{l,j}  \|_{\Lr}^2 + (1+h^{-2}) \| \bbc_{l,j} \|_{\Lr}^2, \\
&\leq Ch^{-2}( \| \bb_{l}  \|_{\Lr}^2 + \| \bbc_{l} \|_{\Lr}^2), \\
&\leq Ch^{-2}(CH^2\| \bv_l \|_{\Lr} + CH^2\| \bv_l \|_{\Lambda}) &&\tn{ by Lemma~\ref{lem2}, } \\
&\leq C\| \bv_l \|_{\Lambda}.
\end{align*}
Note that we are using $h$ to represent the meshsize of mesh level $l$ and 
$H$ to represent that of mesh level $l-1$.
This completes the proof of Theorem~\ref{main}.

\section{Numerical Results} \label{numerics}

In this section, we will report numerical results that supports the theory presented in this paper.
In particular, we present convergence rates for the mixed problem (\ref{mixed_disc}) and the convergence rates for the multigrid V-cycle when applied to (\ref{problem}).
Uniform meshes are used for all examples.

The computer implementation of (\ref{mixed_disc}) is done in the usual way. 
In Table~\ref{ex1}, we report the $L^2_r(\Omega)$-norm of the observed errors when $k=1$.
In Table~\ref{ex2}, we do the same for $k=2$.
In both cases, we use the square domain $\Omega=[0,1]^2$ and choose the right hand side data function $f$ so that
the exact solution $(\bz, p)\in \bH_r(\dive\nolimits^k, \Omega)\times H_r( \grad\nolimits^k, \Omega)$ is
\begin{align*}
p  &= \sin(\pi z)\cos(0.5\pi r)r^2, \\
\bz &=  \left[ {\begin{array}{c}
     -2r\cos(0.5\pi r)\sin(\pi z) + 0.5r^2\pi\sin(0.5\pi r)\sin(\pi z) \\
   -kr\sin(\pi z)\cos(0.5\pi r)  \\
   -r^2\pi\cos(0.5\pi r)\cos(\pi z)
  \end{array} } \right].
\end{align*}
These results are consistent with Theorem~\ref{error}. As for $\left\| p-p_h \right\|_{L^2_r(\Omega)}$, we note that
\[
\left\| p-p_h \right\|_{L^2_r(\Omega)} \leq \left\| p-\Pi_h^S p \right\|_{L^2_r(\Omega)} + \left\| \Pi_h^S p - p_h \right\|_{L^2_r(\Omega)}
\]
and that the first term on the right-hand-side which is of $\mathcal{O}(h)$ dominates even though the second term is of $\mathcal{O}(h^2)$.
\begin{table} 
\caption{Mixed Problem Convergence Rates for Fourier-mode $k=1$}
\label{ex1} 
\begin{center}
 \begin{tabular}{|c| c | c| c| c | c| c |} 
 \hline
 MeshLevel & $\left\| \bz-\bz_h \right\|_{L^2_r(\Omega)}$ & rate & $\left\| p-p_h \right\|_{L^2_r(\Omega)}$  & rate &  $\left\| \Pi_h^S p-p_h \right\|_{L^2_r(\Omega)}$ & rate   \\ [0.5ex] 
 \hline
1  & 2.6827e-01  &          & 4.7997e-02  &         & 1.5531e-02  &         \\
2  & 2.3739e-01  & 0.18  & 4.0432e-02  & 0.25 & 1.5324e-02  & 0.02 \\
3  & 1.4887e-01  & 0.67  & 2.4285e-02  & 0.74 & 5.8420e-03  & 1.39 \\
4  & 7.9202e-02  & 0.91  & 1.2563e-02  & 0.95 & 1.6590e-03  & 1.82 \\ 
5  & 4.0238e-02  & 0.98  & 6.3259e-03  & 0.99 & 4.2928e-04  & 1.95 \\
6  & 2.0200e-02  & 0.99  & 3.1682e-03  & 1.00 & 1.0827e-04  & 1.99 \\
7  & 1.0110e-02  & 1.00  & 1.5847e-03  & 1.00 & 2.7126e-05  & 2.00 \\
8  & 5.0564e-03  & 1.00  & 7.9244e-04  & 1.00 & 6.7853e-06  & 2.00 \\
 \hline
\end{tabular} 
\end{center}
\end{table}

\begin{table} 
\caption{Mixed Problem Convergence Rates for Fourier-mode $k=2$}
\label{ex2} 
\begin{center}
 \begin{tabular}{|c| c | c| c| c | c| c |} 
 \hline
 MeshLevel & $\left\| \bz-\bz_h \right\|_{L^2_r(\Omega)}$ & rate & $\left\| p-p_h \right\|_{L^2_r(\Omega)}$  & rate &  $\left\| \Pi_h^S p-p_h \right\|_{L^2_r(\Omega)}$ & rate   \\ [0.5ex] 
 \hline
1  &  2.7520e-01 &          & 4.6388e-02  &          & 9.4515e-03  &   \\
2  & 2.4466e-01  & 0.17  & 3.9077e-02  & 0.25  & 1.1274e-02  & -0.25 \\
3  & 1.5192e-01  & 0.69  & 2.4055e-02  & 0.70  & 4.7994e-03  & 1.23 \\
4  & 8.0665e-02  & 0.91  & 1.2533e-02  & 0.94  & 1.4118e-03  & 1.77 \\
5  & 4.0967e-02  & 0.98  & 6.3221e-03  & 0.99  & 3.6878e-04  & 1.94 \\ 
6  & 2.0565e-02  & 0.99  & 3.1677e-03  & 1.00  & 9.3232e-05  & 1.98 \\ 
7  & 1.0293e-02  & 1.00  & 1.5847e-03  & 1.00  & 2.3373e-05  & 2.00 \\ 
8  & 5.1476e-03  & 1.00  & 7.9244e-04  & 1.00  & 5.8474e-06  & 2.00 \\
 \hline
\end{tabular} 
\end{center}
\end{table}

Next we verify the uniform convergence rate for the multigrid V-cycle algorithm applied to (\ref{problem}).
We consider two different domains here: the unit square $\Omega_1=[0,1]^2$ and the L-shape domain $\Omega_2 = [0,1]^2\backslash [0.5,1]^2$.
For this example, we choose $f=0$. The initial value $\bx_0$ was chosen randomly, and the stopping criteria was given by $\| \bx_n \|_{\Lambda}/\|\bx_0\|_{\Lambda}<10^{-7}$,
where $\bx_n$ denotes the $n$-th iteration of the multigrid V-cycle. The order of convergence was computed by taking the average of $\| \bx_n \|_{\Lambda}/\|\bx_{n-1}\|_{\Lambda}$.
The prolongation matrix and the restriction matrix is implemented in the usual way.
Tables~\ref{ex3} and \ref{ex4} reports convergence rate for $\Omega_1$ and $\Omega_2$ respectively for various Fourier-modes denoted by $k$.
It is clear that the convergence rate is bounded uniformly as proved in Theorem~\ref{main} for convex domains. This is noticeable even for the L-shape domain $\Omega_2$ suggesting that 
the theory can be extended to non-convex domains.

\begin{table}
\caption{Multigrid V-cycle Convergence Rates for Square Domain $\Omega_1$} \label{ex3}
\begin{center}
\begin{tabular}{|c||c|c|c|c|}
\hline
MeshLevel & $k=1$ &  $k=2$ &  $k=-1$ &  $k=-2$ \\
\hline
2  &  0.12 & 0.06 & 0.11 & 0.07 \\
3  &  0.13 & 0.12 & 0.13 & 0.09 \\
4  &  0.23 & 0.19 & 0.25 & 0.19 \\
5  &  0.24 & 0.23 & 0.27 & 0.22 \\
6  &  0.25 & 0.24 & 0.27 & 0.22 \\
7  &  0.25 & 0.23 & 0.27 & 0.21 \\
8  & 0.24 & 0.21 & 0.26 & 0.18 \\
9  &  0.24 & 0.21 & 0.26 & 0.18 \\
 \hline
\end{tabular}
\end{center}
\end{table}

\begin{table}
\caption{Multigrid V-cycle Convergence Rates for L-Shape Domain $\Omega_2$} \label{ex4}
\begin{center}
\begin{tabular}{|c||c|c|c|c|}
\hline
MeshLevel & $k=1$ &  $k=2$ &  $k=-1$ &  $k=-2$ \\
\hline
2 & 0.13 & 0.16 & 0.16 & 0.16 \\
3 & 0.19 & 0.13 & 0.19 & 0.13 \\
4 & 0.25 & 0.24 &  0.25 & 0.25 \\
5 & 0.27 & 0.25 & 0.28 & 0.25 \\
6 & 0.29 & 0.26 & 0.29 & 0.26 \\
7 & 0.30 & 0.26 & 0.30 & 0.26 \\
8 & 0.31 & 0.27 & 0.32 & 0.26 \\
  \hline
\end{tabular}
\end{center}
\end{table}

\section{Concluding Remarks}

We have provided a multigrid algorithm that can be applied to axisymmetric H(div)-problems with general data
by using a recently developed Fourier finite element space for Fourier modes $|k|\geq 1$. 
Under the assumption that the axisymmetric domain is convex, we proved that the multigrid V-cycle
will converge uniformly with respect to the meshsize. 
Multigrid analysis for axisymmetric H(div)-problems for Fourier mode $k=0$ along with
axisymmetric H(curl)-problems with non-axisymmetric data remain as future work.

\section{Appendix: Proof of Theorem~\ref{cc}} \label{appendix}

We prove the theorem by showing that 
\[
\inf_{\bv_h\in \bB_h}\left\| \bv - \bv_h \right\|_{L^2_r(\Omega)} \leq Ch(|v_\theta|_{H^1_r(\Omega)} + \|kv_z\|_{\tilde{H}^1_r(\Omega)}  +  \|kv_r+v_\theta\|_{\tilde{H}^1_r(\Omega)}).
\]
To do so, we first recall the weighted Clement operator $\bar{\Pi}_h$
constructed in \cite[section 4.3]{BBD:2006}. 
For each vertex $a_j$ in $\TT_h$, associate an arbitrary triangle $K_j$ that contains $a_j$, and let
\[
P_1=\{ar+bz+c: a,b,c\in \RRR \}.
\] 
Then define 
$\bar{\Pi}_j: L^2_r(K_j) \rightarrow P_1(K_j)$ to be the $L^2_r$-orthogonal projection.
The weighted Cl\'ement operator 
$\bar{\Pi}_h: L^2_r(\Omega) \rightarrow V_h :=\{v\in H^1_r(\Omega): v|_K\in P_1 \tn{ for all } K\in \TT_h  \}$ is defined as
\begin{equation} \label{weighted}
\bar{\Pi}_h u = \sum_{j=1}^{N_a}[(\bar{\Pi}_j(u))(a_j)]\phi_j
\end{equation}
where $N_a$ is the number of vertices in $\TT_h$ and $\phi_j$ denotes the lowest order Lagrangian finite element basis function  associated with $a_j$.
The following error estimate follows from \cite[Theorem 1]{BBD:2006}:
\begin{equation} \label{good}
\| u-\bar{\Pi}_hu \|_{L^2_r(K)} \leq Ch_K|u|_{H^1_r(\bar{\Delta}_K)},
\end{equation}
where $\bar{\Delta}_K$ denotes the union of all triangles that shares at least a common vertex with $K$.
We will write $\sigma_a$ to denote the local degrees of freedom associated with (\ref{weighted}).

Next, let
\[
\bH_r(\curl,\Omega)=\{(v_r,v_z)\in \bL^2_r(\Omega): \partial_zv_r-\partial_rv_z\in L^2_r(\Omega) \},
\]
and recall the lowest order \Nedelec space:
\[
\bN_h = \{\bv\in \bH_r(\curl,\Omega): \bv|_K \in ND_1(K) \tn{ for all } K\in\TT_h  \}, 
\]
where 
\[
ND_1=\{(b-az, c+ar): a, b, c\in \RRR \}.
\] 
We now construct a weighted Cl\'ement type operator $\Pi_h: [L^2_{r^3}(\Omega)]^2 \rightarrow \bN_h$ 
that satisfies
\[
\| \tilde{\bv}-\Pi_h\tilde{\bv} \|_{L^2_{r^3}(\Omega)} \leq Ch| \tilde{\bv} |_{H^1_{r^3}(\Omega)}.
\]
We will modify the construction of the basic Clement operator in \cite[section 4.3]{BBD:2006}.
For each edge $e_i$ in $\TT_h$, associate an arbitrary triangle $K_i$ that contains $e_i$.
 Then define 
$\Pi_i: L^2_{r^3}(K_i) \rightarrow ND_1(K_i)$ as the $L^2_{r^3}$-orthogonal projection, i.e.,
\[
(\Pi_i\tilde{\bv}, \bq)_{r^3} = (\tilde{\bv}, \bq)_{r^3} \quad\tn{ for all } \bq\in ND_1(K_i).
\]
The weighted Cl\'ement type operator $\Pi_h: L^2_{r^3}(\Omega) \rightarrow \bN_h$ is defined as 
\begin{equation} \label{degf}
\Pi_h\tilde{\bv} = \sum_{i=1}^{N_e} [\int_{e_i} \Pi_i\tilde{\bv} \cdot \bt ds] \bpsi_{e_i}
\end{equation}
where $N_e$ denotes the number of edges in $\TT_h$, and
$\bpsi_{e_i}$ is the usual \Nedelec basis associated with the edge $e_i$.
We will use $\sigma_e$ to denote the local degrees of freedom corresponding to $(\ref{degf})$,
and $\Delta_K$ to denote the union of all triangles that share an edge with triangle $K$ in $\TT_h$.
We show that $\Pi_h$  satisfies the following lemma.
\begin{lemma} \label{star2}
For all $K\in \TT_h$ and $\tilde{\bv}\in [H_{r^3}^1(\Delta_K)]^2$, we have
\[
\| \tilde{\bv} - \Pi_h\tilde{\bv} \|_{L^2_{r^3}(K)} \leq Ch| \tilde{\bv} |_{H^1_{r^3}(\Delta_K)}.
\]
\end{lemma}
\begin{proof}
First of all, we note the following results that follows by direct calculation for any triangle $K$:
\begin{equation} \label{d1}
\| \bpsi_{e} \|_{L^2_{r^3}(K)}^2 \leq Cr_K^3
\end{equation}
where $r_K = max_{\bx\in K} r(\bx)$. 

Fix a triangle $K$ in $\TT_h$ and let $e_{1}, e_{2},$ and $e_3$ denote its three edges.
Then
\begin{equation} \label{d4}
\Pi_h\tilde{\bv} |_K = 
 \sum_{i=1}^3  (\int_{e_i} \Pi_i\tilde{\bv} \cdot \bt ds  ) \bpsi_{e_i}.
\end{equation}
For each $i$, we have
\begin{equation} \label{d0}
\|  (\int_{e_i} \Pi_i\tilde{\bv} \cdot \bt ds  ) \bpsi_{e_i} \|_{L^2_{r^3}(K)} \leq Ch_{K_i}\| \Pi_i \tilde{v} \|_{L^\infty(K_i)}\|\bpsi_{e_i} \|_{L^2_{r^3}(K)}.
\end{equation} 
If $K_i$ is a triangle that does not intersect $\Gamma_0$, then we can use $\dfrac{r_{max}(K)}{r_{min}(K_i)}\leq C$, (\ref{d1}), and a standard inverse inequality to obtain   
\begin{equation} \label{d3}
\|  (\int_{e_i} \Pi_i\tilde{\bv} \cdot \bt ds  ) \bpsi_{e_i} \|_{L^2_{r^3}(K)}\leq C\| \tilde{\bv} \|_{L^2_{r^3}(K_i)}.      
\end{equation}
If $K_i$ is a triangle that intersects $\Gamma _0$, then 
$r_K \leq Ch_K$ so (\ref{d1}) becomes
\begin{equation*} 
 \| \bpsi_{e} \|_{L^2_{r^3}(K)}^2 \leq Ch_K^3,
\end{equation*}
and 
\[
\| \Pi_i\tilde{\bv} \|_{L^\infty(K_i)} \leq Ch_{K_i}^{-5/2}\| \Pi_i\tilde{\bv} \|_{L^2_{r^3}(K_i)}
\]
by a standard scaling argument, so (\ref{d3}) continues to hold.
Therefore, by (\ref{d4}) and (\ref{d3}), we have
\begin{equation} \label{d5}
\| \Pi_h\tilde{\bv} \|_{L^2_{r^3}(K)} \leq C\| \tilde{\bv} \|_{L^2_{r^3}(\Delta_K)}.
\end{equation}

Since (\ref{r3}) still holds when $\Delta_e$ is replaced by $\Delta_K$, we have
\begin{equation} \label{r3_2}
\inf_{q \in P_0(\Delta_K)} \|v-q\|_{L^2_{r^3}(\Delta_K)} \leq Ch|v|_{H^1_{r^3}(\Delta_K)},
\end{equation}
where in (\ref{r3_2}), we are using $h$ to denote the diameter of $\Delta_K$.

Now, let $\bq\in P_0(\Delta_K)^2$. Then, for all edges $e_i$ of $K$, $\Pi_i \bq$ is equal to $\bq$, so the restriction of $\Pi_h \bq$ onto $K$ is also equal to $\bq$. Therefore,
\begin{align*}
\| \tilde{\bv} - \Pi_h\tilde{\bv} \|_{L^2_{r^3}(K)} &= \| \tilde{\bv} - \Pi_h\tilde{\bv} + \Pi_h \bq - \bq \|_{L^2_{r^3}(K)}, \\
&\leq \| \tilde{\bv} - \bq \|_{L^2_{r^3}(K)} +  \| \Pi_h (\tilde{\bv} - \bq) \|_{L^2_{r^3}(K)}, \\ 
&\leq C\| \tilde{\bv} - \bq \|_{L^2_{r^3}(\Delta_K)} &&\tn{ by (\ref{d5}), } \\  
&\leq Ch| \tilde{\bv} |_{H^1_{r^3}(\Delta_K)} &&\tn{ by (\ref{r3_2}).}
\end{align*}
This completes the proof of the Lemma.
\end{proof}

Let us now show that
\[
\inf_{\bv_h\in \bB_h}\left\| \bv - \bv_h \right\|_{L^2_r(\Omega)} \leq Ch(|v_\theta|_{H^1_r(\Omega)} + \|kv_z\|_{\tilde{H}^1_r(\Omega)}  +  \|kv_r+v_\theta\|_{\tilde{H}^1_r(\Omega)}).
\]
Given $\bu$, define $\tilde{\Pi}_h^{c,k}\bu\in \bB_h$ locally in the following way:
\[
\tilde{\Pi}_h^{c,k}\bu|_K = 
     \sum_{i=1}^3 
     \sigma_{a_i}(u_\theta)
     \begin{pmatrix}
       -\dfrac{1}{k}\phi_i \\
       \phi_i \\
        0          
     \end{pmatrix}  +
     \sum_{i=1}^3
     \sigma_{e_i}(\begin{pmatrix}
       \frac{ku_r + u_\theta}{r} \\
      \frac{ku_z}{r}     
       \end{pmatrix})
         \begin{pmatrix}
        \dfrac{r}{k}\psi_i^r \\
        0 \\
        \dfrac{r}{k}\psi_i^z    
       \end{pmatrix}.
\]
We used $\bpsi_i=(\psi_i^r,\psi_i^z)^T$ to denote $\bpsi_{e_i}$ for simplicity.
The interpolation operator used in \cite{O:2015} are different from $\tilde{\Pi}_h^{c,k}$ in two ways: 
in \cite{O:2015}, the standard nodal interpolation operator was used for $u_\theta$ instead of (\ref{weighted}),
and the standard lowest order \Nedelec interpolation operator was used for $(\frac{ku_r + u_\theta}{r}, \frac{ku_z}{r})^T$ 
instead of (\ref{degf}).


We will show that $\tilde{\Pi}_h^{c,k}$ satisfies the following error estimate:
\[
\left\| \bv - \tilde{\Pi}_h^{c,k}\bv \right\|_{L^2_r(\Omega)} \leq Ch(|v_\theta|_{H^1_r(\Omega)} + \|kv_z\|_{\tilde{H}^1_r(\Omega)}  +  \|kv_r+v_\theta\|_{\tilde{H}^1_r(\Omega)}).
\]
This is true, since
\begin{equation*} 
\begin{aligned}
&\left\| \bv - \tilde{\Pi}_h^{c,k}\bv \right\|_{L^2_{r}(K)}^2 \\
&= \left\| v_\theta -   \sum_{i=1}^3\sigma_{a_i} (v_\theta) \phi_i  \right\|^2_{L^2_{r}(K)} 
+\left\| v_z - \sum_{i=1}^3\sigma_{e_i} (\begin{pmatrix}
       \frac{kv_r + v_\theta}{r} \\
      \frac{kv_z}{r}     
       \end{pmatrix}) \dfrac{r}{k}\psi_i^z  \right\|^2_{L^2_{r}(K)}  \\
&+ \left\| v_r + \sum_{i=1}^3\sigma_{a_i} (v_\theta) \dfrac{1}{k}\phi_i -   \sum_{i=1}^3\sigma_{e_i} (\begin{pmatrix}
       \frac{kv_r + u_\theta}{r} \\
      \frac{kv_z}{r}     
       \end{pmatrix}) \dfrac{r}{k}\psi_i^r   \right\|^2_{L^2_{r}(K)}, \\
&\leq Ch^2|v_\theta|^2_{H^1_r(\bar{\Delta}_K)} +\left\| \dfrac{r}{k}(\dfrac{kv_z}{r} - \sum_{i=1}^3\sigma_{e_i} (\begin{pmatrix}
       \frac{kv_r + v_\theta}{r} \\
      \frac{kv_z}{r}     
       \end{pmatrix}) \psi_i^z)  \right\|^2_{L^2_{r}(K)}  \\
&+ \left\| \dfrac{1}{k}(-v_\theta + \sum_{i=1}^3\sigma_{a_i} (v_\theta) \phi_i) + \dfrac{r}{k}(\dfrac{kv_r + v_\theta}{r} -   \sum_{i=1}^3\sigma_{e_i} (\begin{pmatrix}
       \frac{kv_r + v_\theta}{r} \\
      \frac{kv_z}{r}     
       \end{pmatrix}) \psi_i^r)   \right\|^2_{L^2_{r}(K)} \quad\tn{ by (\ref{good}), } \\
&= C (h^2|v_\theta|^2_{H^1_r(\bar{\Delta}_K)} + \left\| \dfrac{kv_z}{r} - \sum_{i=1}^3\sigma_{e_i} (\begin{pmatrix}
       \frac{kv_r + v_\theta}{r} \\
      \frac{kv_z}{r}     
       \end{pmatrix}) \psi_i^z  \right\|^2_{L^2_{r^3}(K)} + \left\| \dfrac{kv_r +v_\theta}{r} -   \sum_{i=1}^3\sigma_{e_i} (\begin{pmatrix}
       \frac{kv_r + v_\theta}{r} \\
      \frac{kv_z}{r}     
       \end{pmatrix}) \psi_i^r  \right\|^2_{L^2_{r^3}(K)}),  \\
&\leq Ch^2(|v_\theta|^2_{H^1_r(\bar{\Delta}_K)} + | \dfrac{kv_z}{r}  |_{H^1_{r^3}(\Delta_K)}^2 + |\dfrac{kv_r +v_\theta}{r} |_{H^1_{r^3}(\Delta_K)}^2) \quad\tn{ by Lemma~\ref{star2}, } \\
&\leq Ch^2(|v_\theta|^2_{H^1_r(\bar{\Delta}_K)} + \|kv_z\|^2_{\tilde{H}^1_r(\Delta_K)}  +  \|kv_r+v_\theta\|^2_{\tilde{H}^1_r(\Delta_K)}).
\end{aligned}
\end{equation*}
This completes the proof of the theorem.

\bibliographystyle{siam}	
\bibliography{reference}	

\end{document}